\documentclass[letterpaper, 11pt]{article}
\usepackage{amssymb}
\usepackage{amsmath}
\usepackage{amsfonts}
\usepackage{amsthm}
\usepackage[usenames]{color}
\usepackage{graphicx}
\usepackage{dsfont}
\graphicspath{{./Images/}}
\usepackage{verbatim}
\usepackage{mathrsfs}
\usepackage{authblk}
\usepackage{caption,subcaption} 
\usepackage{mathrsfs,enumerate,bbm,latexsym,graphicx}
\usepackage{hyperref,dsfont}
\usepackage[export]{adjustbox}

\newtheorem{theorem}{Theorem}[section]
\newtheorem{lemma}[theorem]{Lemma}
\newtheorem{prop}[theorem]{Proposition}
\newtheorem{remark}[theorem]{Remark}

\theoremstyle{definition}
\newtheorem{definition}{Definition}

\newcommand{\Zn}{\boldZ_n}
\newcommand{\lam}{\lambda}
\newcommand{\al}{\alpha}
\newcommand{\boldZ}{\mathbb{Z}}
\newcommand{\boldE}{\mathbb{E}}
\newcommand{\boldV}{\mathbb{V}}
\newcommand{\bR}{\mathbb{R}}
\newcommand{\bZ}{\mathbb{Z}}
\newcommand{\xt}{x_t}
\newcommand{\et}{e_t}

\newcommand{\xti}{x_{t}(i)}
\newcommand{\etij}{e_{t}(i,j)}

\newcommand{\lampl}{\lambda^{+}_{\alpha}}
\newcommand{\lammin}{\lambda^{-}_{\alpha}}
\newcommand{\s}{\hspace{2mm}}

\newcommand{\bt}{\bullet}
\newcommand{\br}{\blacktriangleright}
\newcommand{\bl}{\blacktriangleleft}
\newcommand{\Aiil}{A_{2,L}}
\newcommand{\Aiir}{A_{2,R}}
\newcommand{\Aiio}{A_{2,O}}
\newcommand{\Aiis}{A_{2,*}}
\newcommand{\Aiiil}{A_{3,L}}
\newcommand{\Aiiir}{A_{3,R}}
\newcommand{\Aiiio}{A_{3,O}}
\newcommand{\Aiiis}{A_{3,*}}
\newcommand{\Aiv}{A_4}
\newcommand{\Biii}{B_3}

\def\E{\mathbb{E}\,}

\newcommand{\sle}{\preceq}
\newcommand{\sge}{\succeq}

\newcommand{\cX}{\mathcal{X}}
\newcommand{\prob}[1]{\mathbb{P}\left(#1\right)}
\newcommand{\probxo}[1]{\mathbb{P}^{\mathcal{X}_0}\left(#1\right)}
\newcommand{\Lra}{\Leftrightarrow}

\newcommand{\vA}{\mathbf{A}}
\newcommand{\vV}{\mathbf{V}}
\newcommand{\vW}{\mathbf{W}}
\newcommand{\vB}{\mathbf{B}}
\newcommand{\vD}{\mathbf{D}}
\newcommand{\eps}{\epsilon}

\newcommand{\sfA}{\mathsf{A}}
\newcommand{\sfB}{\mathsf{B}}

\newcommand{\sfD}{\mathsf{D}}

\newcommand{\sfG}{\mathsf{G}}
\newcommand{\sfU}{\mathsf{U}}

\def\E{\mathbb{E}\,}

\def\beq{ \begin{equation} }
 \def\eeq{ \end{equation} }
 \def\beqx{ \begin{equation*} }
 \def\eeqx{ \end{equation*} }
 \def\beqa{\begin{eqnarray}}
 \def\eeqa{\end{eqnarray}}
 \def\beqax{\begin{eqnarray*}}
 \def\eeqax{\end{eqnarray*}}

\let\ga=\alpha \let\gb=\beta \let\gc=\gamma  \let\gee=\epsilon
\let\gf=\varphi \let\gh=\eta    \let\gl=\lambda 
     \let\gt=\tau \let\gth=\vartheta
\let\gx=\chi

\title{The effect of avoiding known infected neighbors on the persistence of a recurring infection process}
\author[$\star$]{Shirshendu Chatterjee\footnote{SC was supported by the NSF grant DMS-1812148.}}
\author[$\circ$$\ddag$]{David Sivakoff\footnote{DS and MW were supported by the NSF TRIPODS grant CCF-1740761.}}
\author[$\dagger$$\circ$]{Matthew Wascher}

\affil[$\star$]{Department of Mathematics, City University of New York, City College \& Graduate Center, New York, NY 10031}
\affil[$\circ$]{Department of Statistics, The Ohio State University, Columbus, OH 43210}
\affil[$\ddag$]{Department of Mathematics, The Ohio State University, Columbus, OH 43210}
\date{}

\begin{document}
\maketitle
\begin{abstract}
We study a generalization of the classical contact process (SIS epidemic model) in a directed graph $G$. Our model is a continuous-time interacting particle system in which at every time, each vertex is either healthy or infected, and each oriented edge is either active or inactive. Infected vertices become healthy at rate $1$, and pass the infection along each active outgoing edge at rate $\lambda$. At rate $\alpha$, healthy individuals deactivate each incoming edge from their infected neighbors. We study the persistence time of this epidemic model on the lattice $\bZ$, the $n$-cycle $\Zn$, and the $n$-star graph. We show that on $\bZ$, for every $\alpha>0$, there is a phase transition in $\lambda$ between almost sure extinction and positive probability of indefinite survival; on $\Zn$ we show that there is a phase transition between poly-logarithmic and exponential survival time as the size of the graph increases. On the star graph, we show that the survival time is $n^{\Delta+o(1)}$ for an explicit function $\Delta(\alpha,\lambda)$ whenever $\alpha>0$ and $\lambda>0$. In the cases of $\bZ$ and $\Zn$, our results qualitatively match what has been shown for the classical contact process, while in the case of the star graph, the classical contact process exhibits exponential survival for all $\lambda > 0$, which is qualitatively different from our result. This model presents a challenge because, unlike the classical contact process, it has not been shown to be monotonic in the infection parameter $\lambda$ or the initial infected set.
\end{abstract}

\section{Introduction}
The contact process is a stochastic model for an epidemic process on a graph, $G$, which has received a lot of recent attention~\cite{spread-viruses-internet, VHCan, CS2015,Chatterjee2009,HD2018,Mountfordexp,Mountford2013,S2017}. The classical {\em contact process} has a single parameter, $\lambda$, which controls the infection rate {across each edge of $G$}. At any time, the vertices of $G$ can be either infected or healthy. Each healthy vertex becomes infected at rate $\lambda$ times the number of infected neighbors that it has, while infected vertices become healthy at rate 1. Much is known about this model, especially when $G = \bZ^d$, and when $G$ is a finite random graph (see below for more background). In this paper, we study the {\em contact process with avoidance}, in which, in addition to the classical dynamics, each healthy individual attempts to temporarily deactivate {each of the edges that it shares with its infected neighbors}
 at rate $\alpha$. A deactivated edge becomes active again when the infected neighbor becomes healthy. This avoidance behavior is intended to model the tendency of healthy individuals to try and avoid visibly infected individuals in a population.
 
The main ingredients in many proofs about the classical contact process, and many of its variants that have been studied, are duality and additivity. For rigorous definitions and proofs in the case of the classical contact process, see \cite{Liggett1999}. Informally, duality refers to the existence of a time-reversal process that is Markov; the classical contact process is self-dual. Additivity says that if $x_t^A$ is the infected set of vertices at time $t$ with initially infected set $A$, then for every $A,B\subseteq V$, there exists a coupling between $x_t^A, x_t^B$ and $x_t^{A\cup B}$ such that $x_t^{A \cup B}=x_t^A \cup x_t^B$ for all $t\ge 0$. The contact process with avoidance is not known to possess these properties (we suspect it does not), and this is a notable technical challenge in deriving rigorous results about this process.

Our main results indicate that this model exhibits a phase transition similar to the classical contact process on $\bZ$ and on the cycle $\bZ_n := \bZ/n\bZ$, but with a critical infection parameter that grows linearly in $\alpha$. However, it exhibits drastically different behavior on the star graph with $n$ leaves, where the classical contact process survives exponentially long for any $\lambda>0$, while the contact process with avoidance survives only polynomially long (for every $\lambda, \alpha>0$; the case $\alpha=0$ corresponds to the classical contact process). We note that rigorous results for interacting particle systems that coevolve with the underlying topology, such as the CPA, are still scarce in the literature. We discuss notable examples in Section 1.3.

\subsection{Main results}
Let $G = (\boldV,\boldE)$ be a graph with {vertex set}
$\boldV$ and directed {edge set}
$\boldE$. {Now we formally} define the \textit{contact process with avoidance (CPA)} 
 {$(\cX_t)_{t\ge 0}$ {on the graph $G$}, where $\cX_t = (x_t, e_t)$ takes values in $\{0,1\}^{\boldV}\times\{0,1\}^{\boldE}$}. 
 The state of vertex $i \in \boldV$ at time $t$ is given by $x_t(i) \in \{0,1\}$, where 0 indicates that $i$ is susceptible (healthy) and 1 indicates that $i$ is infected at time $t$. The state of the directed edge $(i,j)\in \boldE$ at time $t$ is given by $\etij \in \{0,1\}$, where $0$ indicates that $(i,j)$ is inactive (blocked) and $1$ indicates that $(i,j)$ is active (open) at time $t$. Given the parameters $\lambda,\alpha\ge0$ governing {the per edge infection and deactivation rates}, 
 the process $(\cX_t)_{t\ge0}$ evolves according to the following {update} rules.
\begin{enumerate}
\item For each $i\in \boldV$, $\xti$ goes from $0 \rightarrow 1$ at rate $\lam \sum_{j\in \boldV} x_t(j)e_t(j,i)\mathds{1}\{(j,i)\in \boldE\}$. 
\item For each $i\in \boldV$, $\xti$ goes from $1 \rightarrow 0$ at rate 1.
\item For each $(i,j)\in \boldE$, $\et(i,j)$ goes from $1 \rightarrow 0$ at rate $\al$ if $x_t(j) = 0$ and $x_t(i) = 1$, and at rate $0$ otherwise. 
\item For each $(i,j)\in \boldE$, $\et(i,j)$ goes from $0 \rightarrow 1$ when $x_t(i)=0$.
\end{enumerate}
{We denote the law of the process $(\cX_t)$ starting with initial condition $\cX_0$ 
by $\mathbb P^{\cX_0}$. For brevity, we abuse notation and identify $x_t\in\{0,1\}^\boldV$ with $\{i: x_t(i)=1\}\subset \boldV$.  }

Consider the {one dimensional} lattice $G = (\boldV, \boldE)$ where $\boldV = \boldZ$ and $\boldE = \{(i,j): |i-j| = 1\}$, and let $\mathscr{X}$ be the collection of initial conditions $\cX_0=(x_0, e_0)$ for which $|x_0| < \infty$ and there is at least one $i \in \bZ$ such that $x_0(i) = 1$ and $e_0(i,i-1)+e_0(i,i+1) > 0$.  For each $\alpha >0$, define the lower and upper critical values for $\lambda$ by
\begin{equation}
\begin{aligned}
\lambda_\alpha^- &:= \inf\{\lambda : \probxo{|\xt| \ge 1 \s \forall t > 0} >0 \textrm{ for some } \cX_0 \in \mathscr{X}\}, \\
\lampl &:= \sup\{\lambda : \probxo{x_t(0) = 1 \textrm{ i.o.}} =0\ \forall \cX_0 \in \mathscr{X} \},
\end{aligned}
\end{equation}
where $\{x_t(0) = 1 \textrm{ i.o.}\}$ is the event that $x_t(0)=1$ for an unbounded collection of times $t$.
We define here the upper critical value in terms of the origin becoming infected at arbitrarily large times. We can define another upper critical value, $\lambda_{\alpha,w}^+$, in terms of a ``weak'' notion of survival of the infection, that is, the existence of infected vertices for all $t$:
\begin{equation}
\begin{aligned}
\lambda_{\alpha,w}^+ &:= \sup\{\lambda : \probxo{|\xt| \ge 1 \s \forall t > 0} =0\ \forall \cX_0 \in \mathscr{X} \}.
\end{aligned}
\end{equation}

Note that $\lambda_{\alpha,w}^+\le \lampl$, since $\{|\xt| > 0 \s \forall t > 0\}\supseteq \{x_t(0)=1\ \mathrm{i.o.}\}$. When $\{|\xt| > 0 \s \forall t > 0\}$ occurs, we say the process \textit{survives weakly}, when $\{x_t(0)=1\ \mathrm{i.o.}\}$ occurs, we say the process \textit{survives strongly}, and when $\{|\xt| = 0 \text{ for some } t>0\}$ occurs, we say the process \textit{dies out}. The classical contact process on $\boldZ$ either dies out or survives strongly~\cite{Liggett1999}. However, \cite{pemantle1992} showed that on trees, the contact process may die out, survive weakly but not strongly or survive strongly, depending on $\lambda$ and the tree structure. A natural open question is whether  $\lammin = \lambda_{\alpha,w}^+ = \lampl$, so that there is a single critical value $\lam_{\al}$ separating extinction and strong survival for the contact process with avoidance on $\mathbb{Z}^d$.  In the case of the classical contact process, this question is answered using duality and additivity,  which the contact process with avoidance may not possess.

We now state our main results.

\begin{theorem}\label{lattice-thm}
Let $G = (\boldV, \boldE)$ where $\boldV = \boldZ$ and $\boldE = \{(i,j): |i-j| = 1\}$ and fix $\al > 0$. Then there exist constants $a_1,a_2>0$, not depending on $\al$, such that  $1 + \al \leq \lammin \leq \lampl \leq a_1 + a_2\al$.
\end{theorem}
Theorem~\ref{lattice-thm} says there is a phase transition {(in $\lambda$)} between almost sure extinction and positive probability of strong survival of the infection on $\boldZ$. Moreover, both the upper and lower critical values are linear in $\al$. {Next we focus on the CPA on the $n$-cycle $\mathbb Z_n$.}

\begin{theorem}\label{ring-thm}
Let $G = (\boldV, \boldE)$, where $\boldV = \boldZ_n$ and $\boldE = \{(i,j): |i-j| = 1{\text{ mod } n}\}$. Fix $\al > 0$ and initial condition $\cX_0=(x_0,e_0)$ such that $x_0(i) = 1 \hspace{2mm} \forall i \in \boldV$ and $e_0(i,j) = 1 \hspace{2mm} \forall (i,j) \in \boldE$. Let $\tau = \inf\{t: |x_t| = 0\}$, the time to extinction. Then there exist constants $C, \gamma>0$ such that for $\lam < 1 + \al$, we have $\prob{\tau > C (\log n)^2} \rightarrow 0$ as $n \rightarrow \infty$, and for $\lam > a_1 + a_2\al$, with $a_1,a_2>0$ from Theorem~\ref{lattice-thm}, we have $\prob{\tau \leq e^{\gamma n}} \rightarrow 0$ as $n \rightarrow \infty$.
\end{theorem}
Theorem~\ref{ring-thm} says there is a phase transition in the order of the limiting survival time on $\Zn$, and the upper and lower critical values are linear in $\al$. {Finally, we consider the CPA on star graphs and find dramatically different behavior.}

\begin{theorem}\label{star-thm}
Let $G = (\boldV, \boldE)$ where $\boldV = \{0,1, \ldots ,n-1\}$ and $\boldE = \{(0,j):j \neq 0\} \cup \{(j,0):j \neq 0\}$, and initial condition $\cX_0$ such that $x_0(i) = 1 \hspace{2mm} \forall i \in \boldV$ and $e_0(i,j) = 1 \hspace{2mm} \forall (i,j) \in \boldE$. Let $\tau_{star} = \inf\{t: |x_t| = 0\}$ be the extinction time of the infection, and define
$$
\Delta = 2\left[(\lambda+\alpha+1)-\sqrt{(\lambda+\alpha+1)^2-4\alpha}\right]^{-1}.
$$
Then there exists $N$ such that 
$$
\lim_{K\to\infty}\inf_{n\ge N}\prob{\frac{1}{K}\left(\frac{n}{\log (n)^4}\right)^{\Delta} \leq \tau_{star} \leq Kn^{\Delta}} = 1.
$$
\end{theorem}
Theorem~\ref{star-thm} says the survival time on the star graph is of polynomial order in $n$ and the exponent depends on $\lam$ and $\al$.

\subsection{Graphical Construction}
One popular technique for analyzing contact process models is the Harris construction, which we define here and use throughout the paper. Consider each edge and each vertex on its time axis, and define events using Poisson processes as follows. Figure~\ref{fig:Harris} gives a graphical example of the Harris construction.

\begin{enumerate}
\item Define a Poisson process with intensity $\lambda$ on each directional edge. Then the waiting time starting from time $s$ until the next arrival along the edge $(j,k)$ is $I(s;j,k) \sim \textrm{Exp}(\lambda)$. These arrivals can be thought of as vertex $j$ attempting to infect vertex $k$, and the infection only occurs if $x_{t-}(j) = 1, x_{t-}(k) = 0,$ and $e_{t-}(j,k) = 1$ just before time $t=s + I(s;j,k)$.
\item Define a Poisson process with intensity $1$ on each vertex. Then the waiting time starting from time $s$ until the next arrival is $r(s;j) \sim \textrm{Exp}(1)$. These arrivals can be thought of as vertex $j$ ``attempting" to recover, with a recover only occuring if $x(j) = 1$ at time $s+r(s;j)-$.
\item Define a Poisson process with intensity $\alpha$ on each directed edge $(j,k)$. Then the waiting time starting from time $s$ until the next arrival is $b(s;j,k) \sim \textrm{Exp}(\alpha)$. These arrivals can be thought of as vertex k ``attempting" to avoid vertex j, and the avoidance only occurs if $x(j) = 1, x(k) = 0,$ and $e(j,k) = 1$ at time $s+b(s;j,k)-$.
\end{enumerate}

\begin{figure}
\includegraphics[width=.9\textwidth, center]{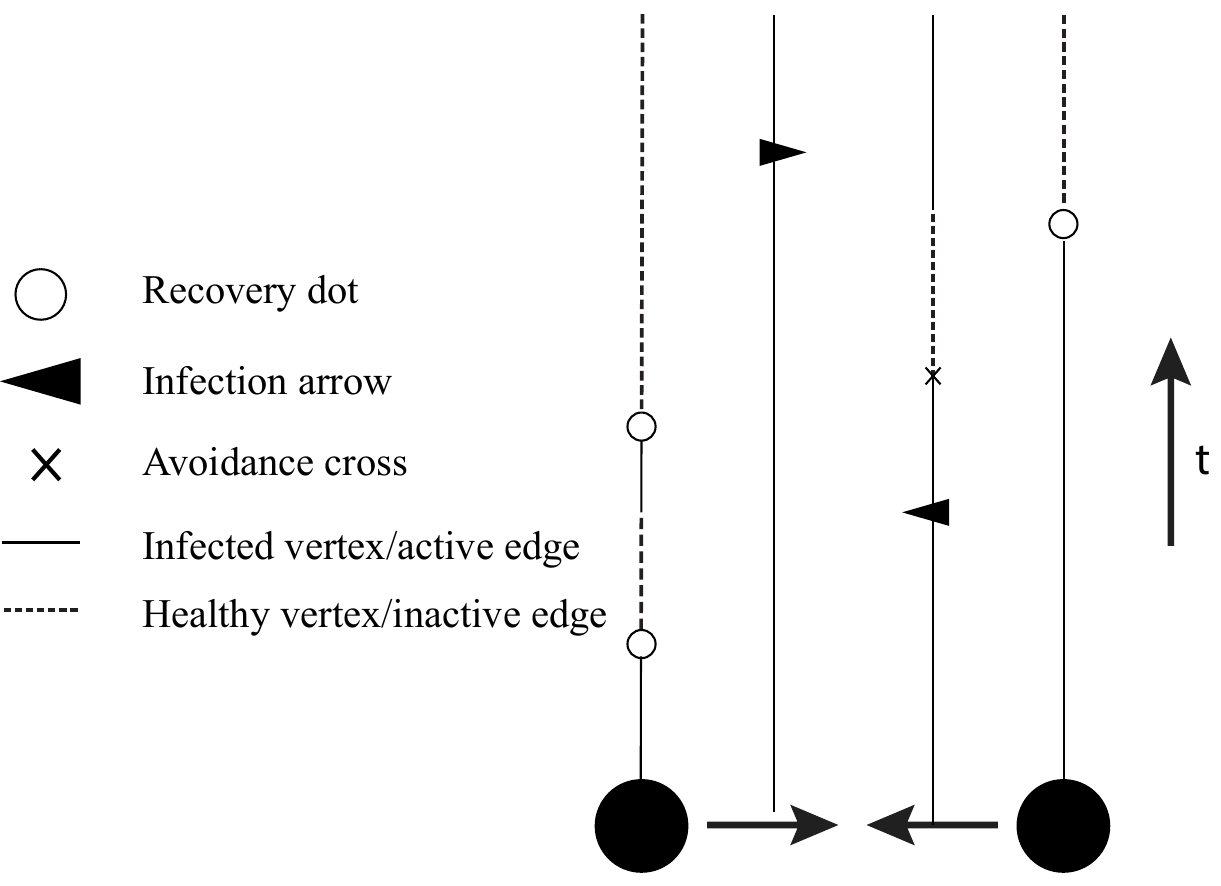}
\caption{\label{fig:Harris} The Harris construction}
\end{figure}

For the classical contact process, the Harris construction provides a coupling of all initial states, which preserves the partial ordering of containment. This monotonicity (or attractiveness) is used to derive many of the known results. The CPA does not appear to possess this kind of monotonicity. Although we do not have a proof of this claim, certainly the Harris construction fails to preserve the partial order on vertex states. An example of the non-monotonicty of the CPA with respect to the set of infected vertices in the Harris construction is shown in Figure~\ref{fig:non-monotone}. Although the initial infected set is larger in the bottom figure, the final infected set is smaller. Nonetheless, this graphical construction of the CPA will be useful in our proofs.

\begin{figure}
\includegraphics[width=\textwidth, center]{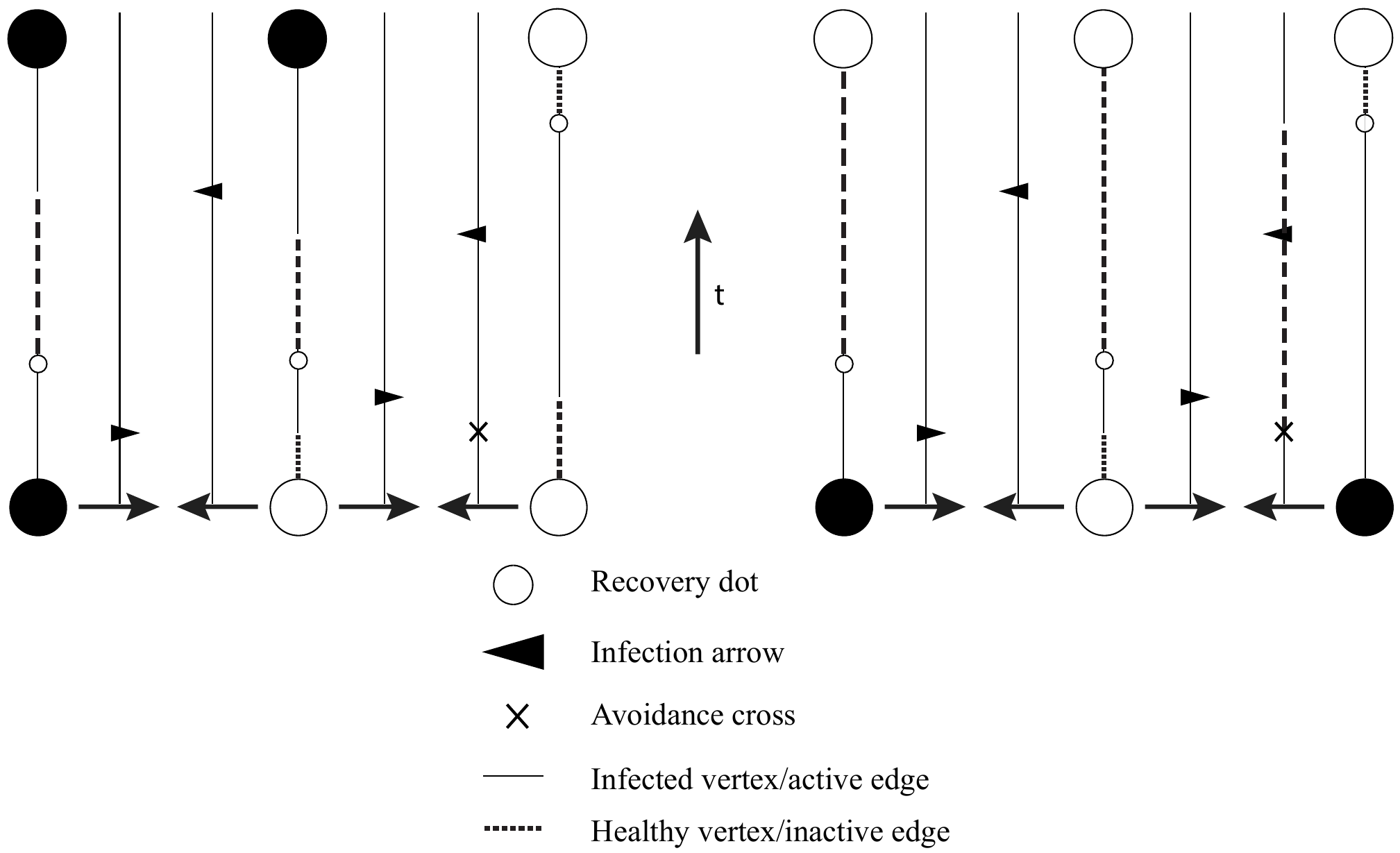}
\caption{\label{fig:non-monotone} The contact process with avoidance is not monotonic in the Harris construction. The event times are the same in each realization, but the additional initially infected vertex in the right realization leads to fewer infected vertices at time $t$.}
\end{figure}

\subsection{Background and related results}
It is well-known that the classical contact process on $\boldZ$ has a critical value $\lam_c>0$, such that when $\lam > \lam_c$ the infection survives forever with positive probability on $\boldZ$ (and has survival time $e^{\Theta(n)}$ on $\Zn$), and when $\lam < \lam_c$ the infection dies almost surely on $\boldZ$ (and survives for $O(\log n)$ time on $\Zn$).  For more on the classical contact process on $\boldZ$ and $\Zn$, see Liggett \cite{Liggett1999}. In contrast, on the star graph and on random graphs having power law degree distributions the limiting survival time is exponential for all $\lambda>0$, and the metastable densities have been derived for a number of models \cite{spread-viruses-internet, VHCan, CS2015,Chatterjee2009, Mountford2013}.

Our model bears resemblance to the adaptive SIS model proposed by \cite{GrossAda}, wherein edges between susceptible and infected individuals are `rewired', rather than deactivated. This model has been of considerable interest in the physics literature \cite{GrossRev, adaptSIS}. Study of this model and its variants has to date been restricted to mean field approximations, moment closures, and simulation results.

Guo, Trajanovsky, van de Bovenkamp, Wang, and Mieghem \cite{GuoEtAl} study a variant of the adaptive SIS model in \cite{GrossAda} more closely related to our contact process with avoidance in which healthy-infected neighbor pairs deactivate the two-way edge between them at rate $\alpha$, and when both vertices are healthy, reactivate the edge at rate $\xi$. They study this model on the complete graph and derive an epidemic threshold using differential equation approximations. Szab\'o-Solticzky, Berthouze, Kiss and Simon~\cite{SBKS:2016} study another variant of the adaptive SIS model where SI edges are deleted at rate $\alpha$ and SS edges are created at rate $\xi$ by independent processes, and study the existence of stable oscillations for this model.

The SIR epidemic, in which infected vertices are removed from the graph upon recovery, has also been studied on evolving graphs. \cite{Jacobsen2016} study a model in which infected vertices are able to activate and deactivate their edges using ODE and pair approximation. \cite{DurrettJunge} study the evoSIR model, in which SI edges rewire at some rate $\alpha$, and find a critical infection rate $\lambda_c$ above which there is positive probability a large epidemic occurs. The long term behavior of the SIR epidemic on evolving graphs tends to be easier to understand than that of the SIS epidemic because in the former case each vertex can become infected at most once.\ 


Remenik \cite{Remenik} proposed an ecologically inspired contact process model, in which sites of $\boldZ$ may become uninhabitable, thereby blocking passage and eliminating infection. His model differs from ours in that the appearance of an uninhabitable site does not depend on the state of neighboring vertices, and the lifetime of uninhabitable sites can be controlled independently of the other process dynamics. This model is monotonic in its parameters when each is viewed individually (although changing multiple together can create incomparable processes, and he proves phase transitions in both the infection rate and decay rate of uninhabitable sites. Jacob and M\"orters \cite{JacobMorters} consider a contact process on evolving scale free networks, and prove that $\lambda_c > 0$ on the evolving graph for certain power-law degree distributions (and sufficiently fast rewiring dynamics) where $\lambda_c = 0$ on the static graph. However, in their model, vertices rewire independently of the state of the graph, and so given the current edges, the future edges are independent of the vertices. This is not true for the contact process with avoidance. Foxall \cite{Foxall} considers an SEIS model on $\boldZ$, in which infected vertices have an incubation period prior to becoming infectious. He claims this model is also not likely to be attractive, and he proceeds to prove existence of a phase transition. However, edges in this model do not evolve. \cite{Durrett1991} study the SIRS epidemic, in which infected vertices enter a removed state for some time after recovery, during which they do not interact at all with other vertices. This model is not monotonic in the usual sense, but their results are limited to the case of $\boldZ^2$ and rely on isoperimetric properties specific to this lattice.

One can also consider similar models with other modes of avoidance. In particular, two other models seem most natural to us in this regard. First, one could consider a model in which infected vertices rather than healthy vertices do the avoiding. This reflects situations in which infected individuals are quarantined to prevent the spread of infection. In this case, an infected node deactivates all edges from itself when it avoids and remains avoiding until it recovers. Another possibility is a model with undirected edges where a healthy vertex avoids an infected vertex by deactivating the bidirectional edge between the two and remains avoiding until the next time both vertices are healthy. In the case of the star graph, we believe that similar results hold for both these alternative models. Applying the heuristic argument we give in section 5 suggests the survival time should still be polynomial in $n$, but with a different power depending on the choice of model. On $\boldZ$ and $\Zn$ however, our upper bound proof techniques do not appear to work for these alternative models.

The proof of Theorem~\ref{ring-thm} only requires the existence in $G$ of a self-avoiding path of length $\Omega(n)$ to conclude exponential survival in $n$ of the CPA. This implies the existence of a supercritical regime on any graph satisfying this condition. Recent work \cite{HD2018, BNNS} has shown that the classical contact process has a subcritical regime only on finite graphs whose degree distributions have exponential tails. These results provide insight into many useful classes of graphs, including power law random graphs and Galton-Watson trees. However, the proofs of exponential survival for all $\lambda > 0$ in the subexponential tails case uses the behavior of the contact process on star graphs as a key ingredient. Because the contact process with avoidance exhibits qualitatively different behavior on stars, whether there exist graphs whose degree distributions have subexponential tails for which the contact process with avoidance has a subcritical regime is an open question. In particular, the cases of power law random graphs and Galton-Watson trees are of interest.


To further explore the phase transition on $\boldZ$ and $\Zn$ we simulated the CPA model for a range of values of $\lam$ and $\al$ on $\Zn$ with $n=500$ vertices. Simulation results appear to indicate that the model is stochastically ordered in $\lam$ for fixed $\al$, in which case a single $\lam_{\al}$ would exist. It also appears that $\lam_{\al}$ is linear in $\alpha$ with a slope between $1.9$ and $2.1$. Figure~\ref{fig:heatmap} shows a survival heatmap for various combinations of $\lam$ and $\al$. We performed 30 iterations of each combination of $\lam$ and $\al$, and the greyscale intensity indicates the proportion of iterations that survived. When $\al = 0$ the simulation identifies that the critical value, which is known to be approximately $1.65$~\cite{Liggett1999}, is between $1.5$ and $1.7$. Simulations with large $\lam$ and $\al$ are expensive, and so we did not simulate as extensively in that case. However, when $\lam = 191.5$ and $\al = 100$ the process appears to die out, while for $\lam = 211.7$ and $\al = 100$ the process appears to survive, which is consistent with a slope between $1.9$ and $2.1$.

\begin{figure}
\includegraphics[scale=.4, center]{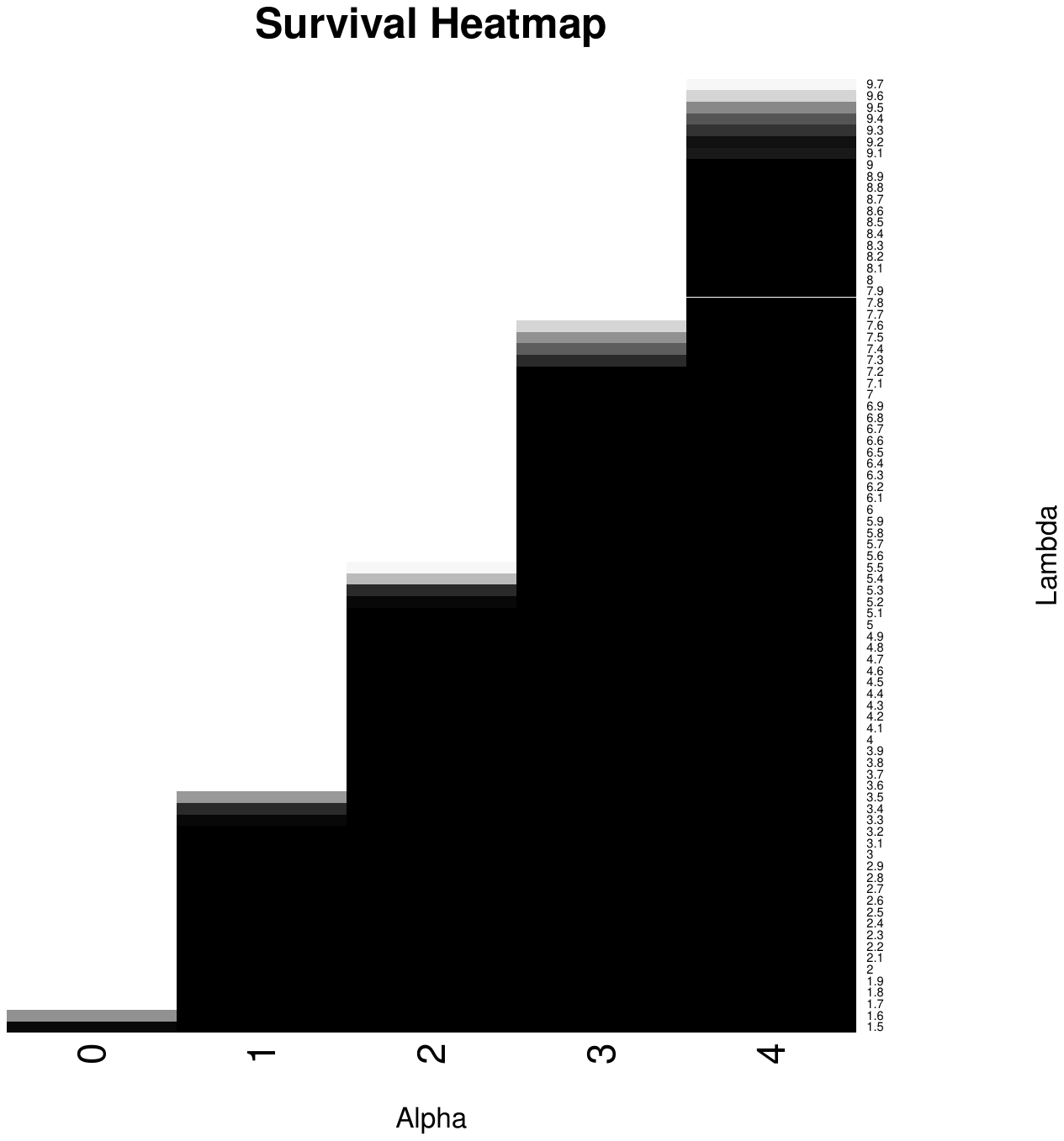}
\caption{\label{fig:heatmap} Survival Heatmap}
\end{figure}

The remainder of the paper is devoted to proving our three main theorems.

\section{Lower Bound for $\lammin$ on $\boldZ$}\label{sec:Z lower bd}

\begin{lemma}\label{RWbound}
Fix $\al > 0$. 
Then $\lammin \geq 1 + \al$.
\end{lemma}

\begin{proof}
Consider $(\cX_t)_{t\ge 0}$ with initial configuration $\cX_0$ where $|x_0| < \infty$. Since $|x_0| < \infty$, $x_0$ must have leftmost and rightmost infected vertices whose locations we will denote by $l_0$ and $r_0$. Let $(l_t)_{t\ge 0}$ and $(r_t)_{t\ge 0}$ track the locations of the leftmost and rightmost infected vertices in $\cX_t$ (with the convention that $l_t=\infty$ and $r_t = -\infty$ if $x_t\equiv 0$.), We now define an embedded discrete time process $(L_s)_{s\in \bZ_+}$ of $(l_t)_{t\ge 0}$ as follows. A step in the chain $L_s$ occurs when either 
\begin{enumerate}
\item Vertex $L_s$ infects vertex $L_s-1$, in which case $L_{s+1} = L_s-1$, or 
\item Vertex $L_s$ recovers, in which case $L_{s+1} = l_{t+}$ where $l_{t+}$ is the new leftmost infected vertex at time $t$ immediately after vertex $L_s$ recovers. 
\end{enumerate}
Now observe that $L_{s+1} < L_s$ will hold only if $L_s$ attempts to infect $L_{s-1}$ before either $L_s$ recovers or $L_{s-1}$ avoids $L_s$. So then for $\lam < 1 + \al$ $$\prob{L_{s+1} < L_s} < 1/2$$
By symmetry we can construct an analogous discrete time process $R_s$ starting from $r_0$ such that $$\prob{R_{s+1} > R_s} < 1/2$$
As long as $\cX_t$ persists we are assured $L_s \leq R_s$. We thus observe that by the first time $R_s < L_s$ the process $\cX_t$ must have reached its absorbing state. By our choice of $\lam < 1+\al, L_s$ and $R_s$ are dominated by random walks with positive and negative drifts respectively and $L_0 \leq R_0$ and so with probability 1 they will eventually cross and $\cX_t$ will have died out.
\end{proof}

\section{Upper Bound for $\lampl$ on $\boldZ$}
For the classical contact process the supercritical regime can be proved by comparison with an oriented percolation process. The idea is to divide up spacetime into nonoverlapping boxes and declare a box ``good'' if the infection can successfully pass through on the time axis. The boxes can then be thought of as sites in an oriented site percolation model, which is known to survive strongly when the occupation probability is sufficiently large \cite{percolation}. If the oriented percolation model is supercritical, then the infection survives strongly by propagating through the good regions with positive probability.

In the case of the classical contact process we know from monotonicity that ``goodness" of regions is positively correlated. Thus, if we can show that a region is good with probability at least $p$ using only events in the part of the Harris construction contained in that region, we can then dominate an oriented site percolation with occupancy probability $p$. However, the contact process with avoidance is not monotonic in the Harris construction, and so we must deal with the dependence among regions in a different way. We do this by finding a uniform bound on the probability that a given region is good regardless of what happens on its spacetime boundary and show this probability can be made arbitrarily close to 1. Section 3 formalizes and proves this assertion.



We begin by defining our regions. Let $\tau \in \bR$ be a fixed timescale, which will be chosen later to depend on $\alpha$. For each $k\in \boldZ$ and integer $\ell \ge 0$ such that $k+\ell$ is even, define the spacetime region $R_{k,\ell} = \{i : 2k\le i \le 2k+3\}\times \{(i,j) : 2k\le i,j \le 2k+3\}\times [\ell\tau,(\ell+1)\tau)$, which is a subset of $\bZ\times\boldE\times \bR_+$. Note that each block, $R_{k,\ell}$, contains 4 vertices and the edges between them over a time interval of length $\tau$. For convenience we will call the vertices in $R_{k,\ell}$ $\{0,1,2,3\}$, and we will consider waiting times to events using the Harris construction defined in Section 1.


We now define some notation to use for diagrams of states of vertices and edges among $\{0,1,2,3\}$.  Let $\bt$ denote an infected vertex, let $\circ$ denote a healthy (susceptible) vertex, and let $?$ denote a vertex that is either healthy or infected. Let $\bl$ denote a blocked right-pointing edge, that is, $e_t(i,i+1) = 0$, so the vertex $i+1$ is avoiding the infected vertex $i$. Similarly, let $\br$ denote a blocked left-pointing edge, and let $\Lra$ indicate that both the left- and right-pointing edges are active (open).  Let $-$ indicate any of the three possible states for the pair of edges between $i$ and $i+1$. Note that under our dynamics, we can never have $e_t(i,i+1) = e_t(i+1,i) = 0$ provided we don't have such a configuration initially. For a region $R_{k,\ell}$:
\begin{enumerate}
\item Let $\Aiil$ denote the states that contain $\bt\Lra\bt\bl\,?-?$.
\item Let $\Aiir$ denote the states that contain $\bt\Lra\bt\br\bt-?$.
\item Let $\Aiio$ denote the states that contain $\bt\Lra\bt\Lra?-?$
\item Let $\Aiis$ denote the union of 1-3 and their reflections across the middle edge.
\item Let $\Aiiil$ denote the states that contain $\bt\Lra\bt\Lra\bt\bl \,?$.
\item Let $\Aiiir$ denote the states that contain $\bt\Lra\bt\Lra\bt\br\bt$.
\item Let $\Aiiio$ denote the states that contain $\bt\Lra\bt\Lra\bt\Lra?$.
\item Let $\Aiiis$ denote the union of 5-7 and their reflections across the middle edge.
\item Let $\Aiv$ denote the states that contain $\bt\Lra\bt\Lra\bt\Lra\bt$.
\end{enumerate}
We call the region $R_{k,\ell}$ \textbf{good} if starting from any of the configurations in $\Aiis$ at time $\ell\tau$ we end in state $\Aiv$ at time $(\ell+1)\tau$. The following lemmas identify a sequence of events in the region $R_{k,\ell}$ such that starting from any initial configuration in $\Aiis$ at time $0$ we reach state $\Aiv$ at time $\tau$ regardless of what happens on the external spacetime boundary of $R_{k,\ell}$, and that for fixed $\al > 0$ this probability can be made arbitrarily close to $1$ with appropriate choices of $\tau$ and $\lambda$.

\begin{figure}
\includegraphics[width=\textwidth]{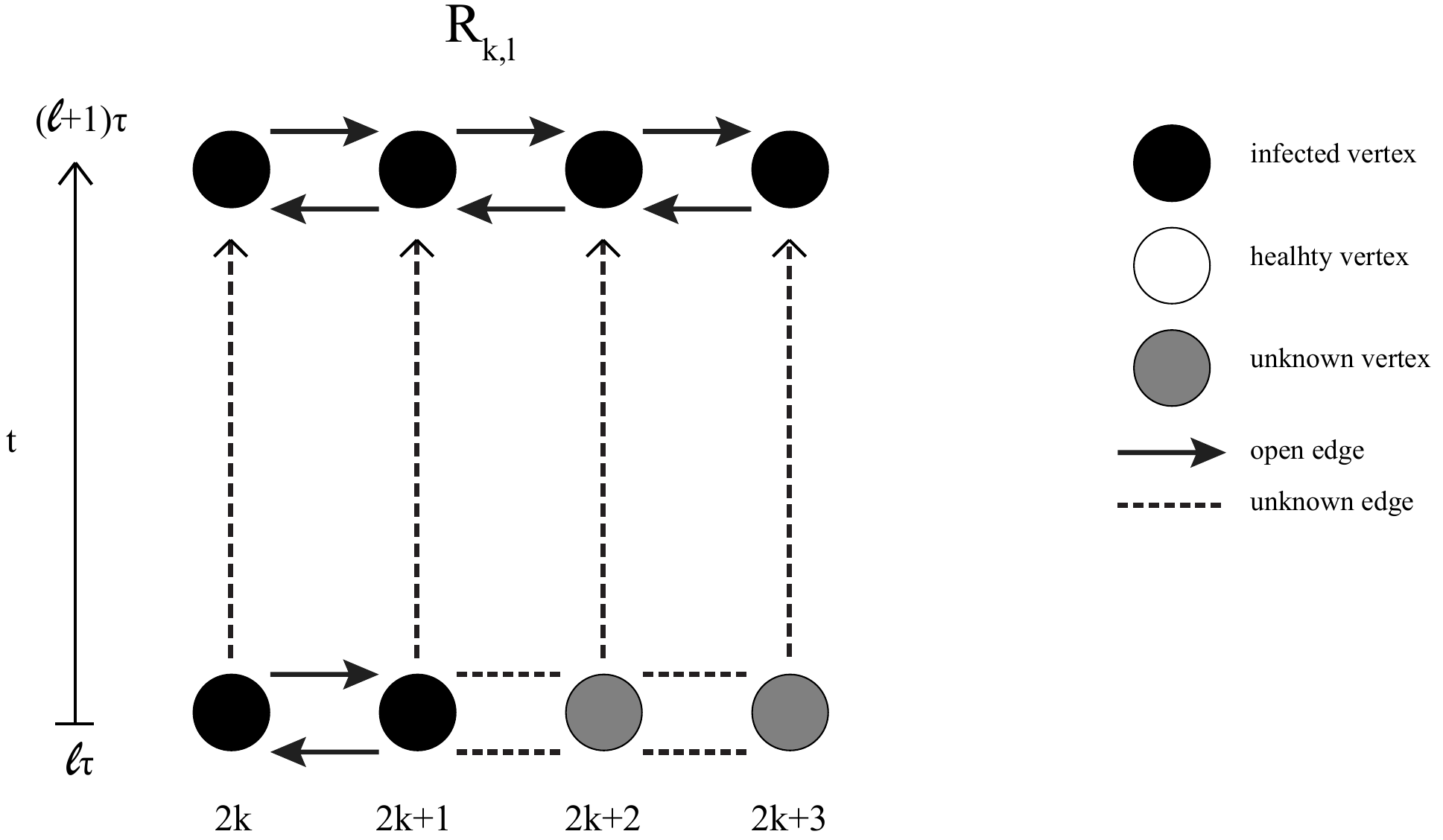}
\caption{A good region.}
\end{figure}

\begin{lemma}\label{A2star}
Fix $\alpha>0$ and $p \in (0,1)$. There exist $\lambda^* = \lambda^*(p,\alpha)$ and $ t = t(p,\alpha)$ such that for any $\lambda\ge \lambda^*$, starting from any configuration in $\Aiis$ on $\{0,1,2,3\}\times \{(i,j) : i,j\in\{0,1,2,3\}\}$ at time $0$, the probability of hitting a state in $\Aiiis$ by time $t$ is at least $1-p$.
\end{lemma}

\begin{proof}
We consider the three cases for initial configurations $\Aiio$, $\Aiil$, and $\Aiir$ separately.  Note that because the process dynamics are symmetric it suffices to consider starting with the left two vertices infected.
\begin{enumerate}
\item $\Aiio:$ 

Suppose the initial configuration is $\cX_0 \in A_{2,O}$. We reach $A_{3,*}$ by time $t$ if vertex $1$ attempts to infect vertex $2$ before time $t$, and this infection event occurs before vertex $0$ recovers, vertex $1$ recovers, or the edge $(1,2)$ becomes inactive. Letting $T_{3,*} = \inf\{s : \cX_s \in A_{3,*}\}$, we have
\begin{equation}
\begin{aligned}
\probxo{T_{3,*} > t} &\le \prob{I(0;0,1) > t} + \prob{I(0;0,1) > \min(r(0;0),r(0;1),b(0;1,2))}\\
&= e^{-\lambda t} + \frac{2+\alpha}{\lambda + 2 + \alpha}.
\end{aligned}
\end{equation}
Choosing $\lambda^* = 2(2+\alpha)/p$ and $t = \frac{1}{\lambda^*}\log(2/p)$ finishes the proof in this case.

%
%

\item $\Aiil:$ Suppose the initial configuration is $\cX_0 \in \Aiil$. We reach $\Aiiis$ by time $t = t_1 + t_2 + t_3$ if vertex 1 recovers before time $t_1$, vertex 0 is infected when this recovery occurs, vertex 0 attempts to infect vertex 1 by time $t_1 + t_2$, this infection occurs before vertex 1 recovers or the edge (1,2) becomes inactive, vertex 1 attempts to infect vertex 2 by time $t_1 + t_2 + t_3$, and this infection occurs before either vertex 0 or 1 recovers or the edge (1,2) becomes inactive. Define the random times $s_1 = \inf\{t > 0: \textrm{vertex 2 recovers}\}$, $s_2 = \inf\{t > s_1: \textrm{vertex 1 infects vertex 2}\}$. Letting $T_{3,*} = \inf\{s:\cX_s \in \Aiiis\}$ we have 
\begin{equation}
\begin{aligned}
\probxo{T_{3,*} > t} \le&  \prob{r(0;1) > t_1} 
\\&+ \prob{\textrm{vertex 0 is healthy when vertex 1 recovers}|r(0;1) \leq t_1} \\&+ \prob{I(s_1;0,1) > t_2} + \prob{I(s_1;0,1) > \min(r(s_1;0), b(s_1;0,1))} \\&+ \prob{I(s_2;1,2) > t_3} + \prob{I(s_2;1,2) > \min(r(s_2;0),r(s_2;1),b(s_2;1,2))}
\end{aligned}
\end{equation}
Choosing $\lam^{**} = 12(2 + \al)/p$, $t_1 = \log(6/p)$, and $t_2 = t_3 = \frac{1}{\lam^{**}}\log(6/p)$ yields
\begin{equation}
\begin{aligned}
\probxo{T_{3,*} > t} &\le 4p/6 + \prob{\textrm{vertex 1 is healthy when 2 recovers}|r(0;2) \leq t_1}
\end{aligned}
\end{equation}
Now given a fixed $t_1$, vertex 0 is healthy when vertex 1 recovers given $r(0;1) \leq t_1$ if for a fixed $k$ vertex 0 recovers at most $k$ times, and whenever vertex 0 recovers vertex 1 successfully reinfects vertex 0 before vertex 1 recovers or the edge (1,0) becomes inactive. The number of recoveries of vertex 0 in time $t_1$ is Poisson($t_1$). Define the random times $v_0 = 0$ and $v_k = \inf{v_k > v_{k-1}: \textrm{vertex 0 recovers}}$. Choose $k$ such that if $X\sim$ Poisson($t_1$) then $\prob{X > k} \leq p/6$ and $\lam \geq$ $6k(1+\al)/p$, then

\begin{equation}
\begin{aligned}
\prob{\textrm{vertex 0 is healthy when vertex 1 recovers}|r(0;1) \leq t_1} & \leq \prob{X > k} 
\\& \hspace{-3cm} + \sum_{m=1}^k \prob{I(v_{k-m};1,0) > \min(r(v_{k-m};0),b(v_{k-m};1,0)}\\
&\leq 2p/6
\end{aligned}
\end{equation}
and choosing $\lam^* = \max\{\lam^{**}, 6k(1+\al)/p\}$ and $t = t_1+t_2+t_3$ finishes the proof in this case.

\begin{figure}
\includegraphics[width=\textwidth]{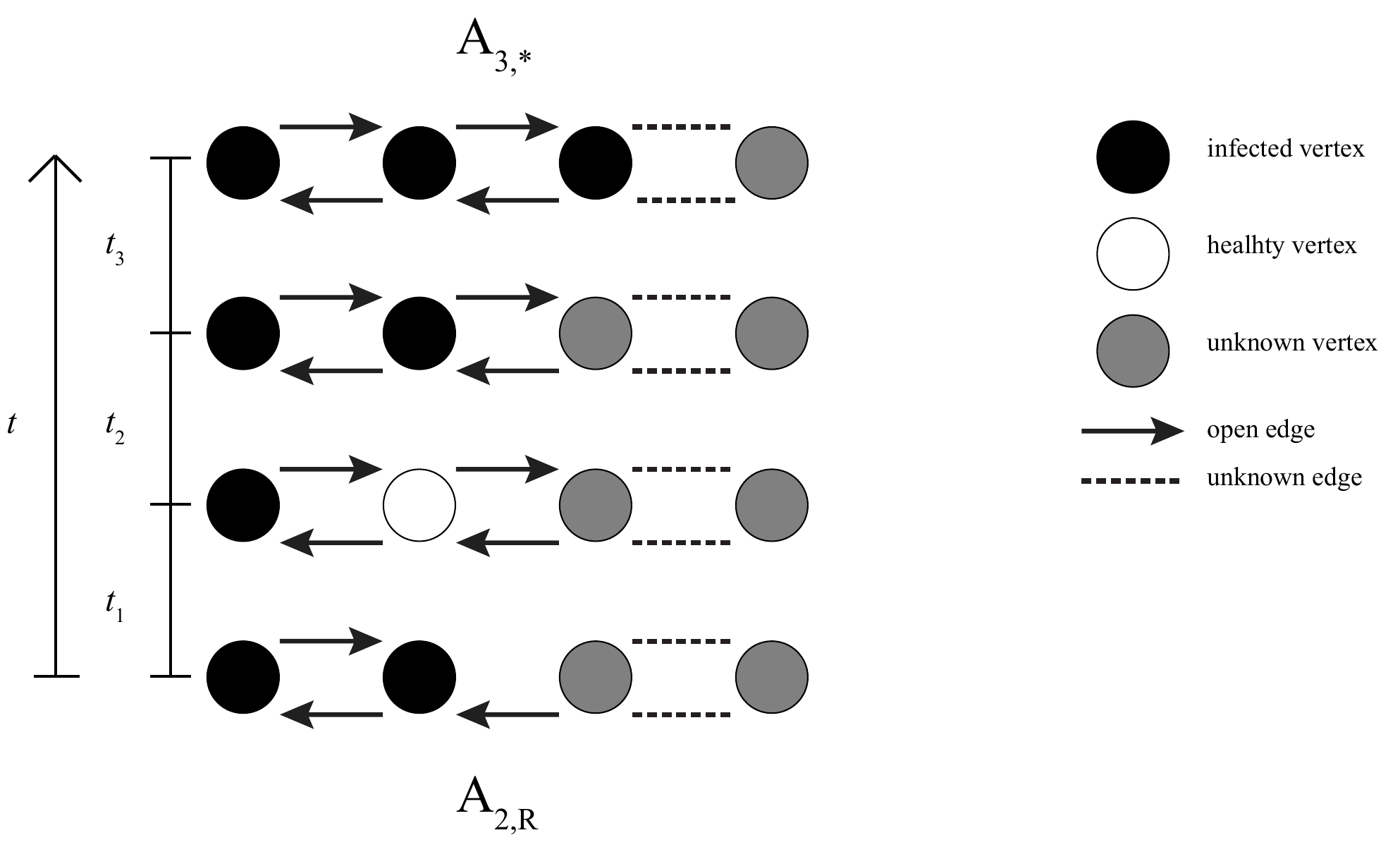}
\caption{A sequence of events leading from $\Aiir$ to $\Aiiis$}
\end{figure}

\item $\Aiir:$ Suppose the initial configuration is $\cX_0 \in \Aiir$. In this case we know vertex 2 must be infected since the edge (1,2) is inactive. We reach $\Aiiis$ by time $t = t_1 + t_2$ if vertex 2 recovers before time $t_1$, vertices 0 and 1 are infected when this recovery occurs, vertex 1 attempts to infect vertex 2 before time $t_1 + t_2$, and this infection occurs before either vertex 0 or vertex 1 recovers or the edge (1,2) becomes inactive. Define the random time $s_1 = \inf\{s > 0: \textrm{vertex 3 recovers}\}$. Letting $T_{3,*} = \inf\{s: \cX_s \in \Aiiis\}$, we have
\begin{equation}
\begin{aligned}
\probxo{T_{3,*} > t} & \le \prob{r(0;3) > t_1}\\
& \hspace{3mm}+ \prob{\textrm{vertex 1 or 2 is healthy when 3 recovers}|r(0;3) \leq t_1}\\
& \hspace{3mm} + \prob{I(s_1;2,3) > t_2} + \prob{I(s_1;2,3) > \min(r(s_1;1), r(s_1;2), b(s_1;2,3))}
\end{aligned}
\end{equation}
Choosing $\lam^{**} = 12(2 + \al)/p$, $t_1 = \log(6/p)$, and $t_2 = \frac{1}{\lam^{**}}\log(6/p)$ yields
\begin{equation}
\begin{aligned}
\probxo{T_{3,*} > t} &\le 3p/6 + \prob{\textrm{vertex 1 or 2 is healthy when 3 recovers}|r(0;3) \leq t_1}
\end{aligned}
\end{equation}
Now given a fixed $t_1$ vertices 0 and 1 are infected when vertex 2 recovers given $r(0;2) \leq t_1$ if for fixed $k_0$ and $k_1$ vertex 0 recovers at most $k_0$ times, vertex 1 recovers at most $k_1$ times, whenever vertex 0 recovers vertex 1 successfully reinfects vertex 0 before vertex 1 recovers or the egde (1,0) becomes inactive, and whenever vertex 1 recovers vertex 0 successfully reinfects vertex 1 before vertex 0 recovers or the edge (0,1) becomes inactive. The number of recoveries of vertex j where $j = 0,1$ in time $t_1$ are independent Poisson($t_1)$. Define the random times $v^0_0 = 0$, $v^0_k = \inf\{v > v^0_{k-1}:\textrm{vertex 0 recovers}\}$, $v^1_0 = 0$, and $v^1_k = \inf\{v > v^1_{k-1}:\textrm{vertex 1 recovers}\}$. Choose $k$ such that if $X \sim$ Poisson($t_1$) then $\prob{X > k} \leq p/6$, $k_1 = k_2 = k$, and $\lam \geq 12k(1 + \al)/p$. Then
\begin{equation}
\begin{aligned}
\prob{\textrm{0 or 1 is healthy when 2 recovers}}
&\leq 2 \prob{X > k} 
\\& \hspace{-2cm} + \sum_{m=1}^{k_1} \prob{I(v^0_{m-1};1,0) > \min(r(v^0_{m-1};1),b(v^0_{m-1};1,0))}\\
& \hspace{-2cm} + \sum_{n = 1}^{k_2} \prob{I(v^1_{n-1};0,1) > \min(r(v^2_{n-1};0),b(v^2_{n-1};0,1))}\\
& \leq 3p/6,
\end{aligned}
\end{equation}
and choosing $\lam^* = \max\{\lam^{**},12/(1+\al)/p\}$ and $t = t_1 + t_2$ finishes the proof in this case.

\end{enumerate}
Finally, if we choose the maximum values of $\lam_*$ and $t$ appearing in 1-3. above, then we simultaneously satisfy all three cases, completing the proof.
\end{proof}

\begin{lemma}\label{A3star}
Fix $\al > 0$ and $p \in (0,1).$ There exist $\lam^* = \lam(p,\al)$ and $t = t(p, \al)$ such that for any $\lam \geq \lam^*$, starting from any configuration in $\Aiiis$ on $\{0,1,2,3\} \times \{(i,j):i,j \in \{0,1,2,3\}\}$ at time 0, the probability of hitting  state $A_{4}$ by time $t$ is at least $1-p$.
\end{lemma}

\begin{proof}
The proof of this lemma follows the same arguments as the proof of the previous lemma by examining the cases in of starting in $\Aiiio, \Aiiil,$ and $\Aiiir$. Because there are more potential recoveries to consider the constants are larger, but nothing else changes.
\end{proof}

Lemmas \eqref{A2star} and \eqref{A3star} show that for fixed $\al > 0$, we can reach $\Aiv$ starting from any configuration in $\Aiis$ within time $\tau$ with arbitrarily large probability for an appropriate choice of $\lam$ and $\tau$. However, for the region $R_{k,\ell}$ to be good, we must be in state $A_4$ at time $\tau$. The following lemmas show that this event also has arbitrarily high probability for appropriately chosen $\lam$ and $\tau$.

\begin{lemma}\label{B3}
Let $\Biii$ consist of all states on $\{0,1,2,3\}$ and the associated edges where any three of $\{0,1,2,3\}$ are infected and all of $\{(i,j):i,j \in \{0,1,2,3\}\}$ are active. Fix $p \in (0,1), t > 0.$ Then there exists $\lam^* = \lam^*_{t,p,\al}$ such that for $\cX_0 \in A_4$ and all $\lam \geq \lam^*$, $$\probxo{\cX_s \cap R_{k,\ell} \in B_3 \cup A_4 \hspace{2mm} \forall s \leq t} \geq 1-p.$$
In words, the probability that starting from $A_4$ at time 0 the process $\cX_s$ in the region $R_{k,\ell}$ only visit states in $B_3$ or $A_4$ in the time interval $s \in [0,t]$ is at least $1-p$.
\end{lemma}

\begin{proof}
For fixed $t$ and starting from $\Aiv$ at time 0, $\cX_s \cap R_{k,\ell}$ only ever visit states in $\Biii$ or $\Aiv$ for $s \in [0,t]$ if for a fixed $k$ each vertex recovers at most $k$ times, and whenever a vertex recovers it becomes reinfected before any of the other vertices recover or any of its edges become inactive. The number of recoveries of each vertex is Poisson($t$) independently. For each $i \in \{0,1,2,3\}$ define the random times $v_0^i = 0, v_k^i = \inf\{v > v^i_{k-1}: \textrm{vertex i recovers}\}$.  Choose $k$ such that if $X \sim$ Poisson($t$) then $\prob{X > k} \leq p/8$ and $\lam^* = 8k(3 + 2\al)/p$. Then for $j,h \in \{0,1,2,3\}$

\begin{equation}
\begin{aligned}
\probxo{\cX_s \cap R_{k,\ell} \in B_3 \cup A_4 \hspace{2mm} \forall s \leq t} & \leq 4\prob{X > k}\\
& \hspace{-3cm} + \sum_{j=0}^3 \sum_{m=1}^k \prob{\min_{|j-h|=1}\{I(v^j_{m-1};h,j)\} > \min \{\min_{h \neq j}\{r(v^j_{m-1};h)\}, \min_{|j-h|=1}\{b(v^j_{m-1};h,j)\}\}}\\
& \leq p
\end{aligned}
\end{equation}
\end{proof}


\begin{lemma}\label{A4}
Suppose that $\cX_0 \in \Aiv \cup \Biii$. Fix $p \in (0,1)$. Then there exist $t = t(p,\alpha)$ and $\lam^* = \lam^*(t,p,\al)$ such that $\probxo{\cX_t \in \Aiv} \geq 1-p$.
\end{lemma}

\begin{proof}
$\cX_t \in \Aiv$ if none of the vertices attempt to recover before time $t$, a neighbor of the uninfected vertex (if $\cX_0 \in \Biii$ and there is one) attempts to infect it before time $t$, and this infection occurs before any edges of the uninfected vertex become inactive. The total number of attempted recovers of all 4 vertices follows a Poisson($4t$) distribution. Let $X \sim$ Poisson($4t$). If we call the (possibly) uninfected vertex $k$, then for $j \in \{1,2,3,4\}$
\begin{equation}
\begin{aligned}
\probxo{\cX_t \notin \Aiv} \leq  \prob{X > 0}
\\& \hspace{2mm} + \prob{\min_{|j-k| = 1} I(0;,j,k) > t}\\
& \hspace{2mm} + \prob{\min_{|j-k| = 1} I(0;j,k) > \min_{|j-k| = 1}b(0;j,k)}
\end{aligned}
\end{equation}
and choosing $t$ such that $\prob{X > 0} \leq p/3$ and $\lam^* = \max\{\lam^{**}, 3(2\al)/p\}$ where $\lam^{**}$ is chosen such that $\int_0^t \lam^{**}e^{-\lam^{**}u}du = 1 - p/3$ completes the proof.
\end{proof}

Taken together, lemmas \eqref{B3} and \eqref{A4} show that if we can reach state $A_4$ within time $\tau_*$, then for large enough $\lambda$ we will be in a state in $B_3 \cup A_4$ at time $\tau_*$ with high probability, and we can choose $t_*$ such that at time $\tau = \tau_* + t_*$ we are in state $A_4$ with high probability.

Lemmas \eqref{A2star}-\eqref{A4} identify a sequence of events that ensures a region $R_{k,\ell}$ is good regardless of what happens on the external spacetime boundary of $R_{k,\ell}$ and shows that for any $p \in (0,1)$ the probability of this sequence can made at least $p$ for an appropriate choice of $\lam$ and $\tau$. Thus, if we choose such $\lam$ and $\tau$, the good regions stochastically dominate an oriented site percolation with occupancy probability $p$ for each site independently. If we take $p = .9$, which is sufficient for the oriented percolation to be supercritical, and verify lemmas 3.1-3.4, we arrive at $\lampl \leq 142557 + 47519\al$.

\begin{remark}
If we only wanted to prove weak survival, then it would be sufficient to define overlapping regions of $3$ vertices and follow the infection in a single direction, applying the results of \cite{Liggett1997} to the resulting dependent percolation. Because the regions are now smaller, this would yield some improvement to the constants in the proof. However, it is not obvious how to obtain strong survival from this construction, and the constants will still be relatively large, so we do not feel the improvement is very meaningful.
\end{remark}

\section{Extension to $\Zn$}

We now consider the contact process with avoidance on $\Zn$. We can adapt the arguments developed in previous two sections and combine them with some known results about oriented percolation to show a phase transition on $\Zn$ in the following sense. For fixed $\al$ and starting from all vertices infected and all edges active, for sufficiently small $\lam$ the typical time to extinction is $O((\log n)^2)$ and for sufficiently large $\lam$ the typical time to extinction is at least $e^{\gamma n}$. We state this more precisely in the following theorem, from which Theorem~\ref{ring-thm} follows immediately.

\begin{theorem}
Fix $\al > 0$, and initial condition $\cX_0$ such that $x_0(i) = 1 \hspace{2mm} \forall i \in \boldV$ and $e_0(i,j) = 1 \hspace{2mm} \forall (i,j) \in \boldE$. Let $\tau = \inf\{t: |x_t| = 0\}$, the time to extinction. Then for $\lam < 1 + \al,$ $\prob{\tau > C (\log n)^2} \rightarrow 0$ as $n \rightarrow \infty$ for some $C > 0$ and for $\lam > 142557 + 47519\al$,  we have $\prob{\tau \leq Ce^{\gamma n}} \rightarrow 0$ for some $C, \gamma > 0$ as $n \rightarrow \infty$.
\end{theorem}

\subsection{Subcritical proof: $\lambda<1+\alpha$}

The proof of the subcritical regime on $\boldZ_n$ is again notably complicated by the fact that the CPA is not an attractive particle system. The basic strategy is to break $\boldZ_n$ into regions and establish that with high probability many regions will be cleared of the infection quickly and stay clear of the infection for a long time. However, as was the case on $\boldZ$, the dependence among regions is complex and substantial work is needed to establish bounding processes that allow us to treat the regions as if they were independent.

Consider the CPA on $\boldZ_n$ starting from all vertices infected and suppose $\lam < 1+\al$. We divide $\boldZ_n$ into regions as follows. Let $C^*>0$ be a constant to be chosen later. The $i$th region consists of the vertices $\{3 C^* (i-1) \log n, \ldots , 3 C^* i \log n - 1\}$, all edges among these vertices, and in addition the edges $(3 C^* (i-1) \log n - 1, 3 C^* (i-1) \log n)$ and $(3 C^* i \log n, 3 C^* i \log n - 1)$. Note that each region contains $3 C^* \log n$ consecutive vertices. Of course the $C^*\log n$ term will generally not be an integer, and our convention will be to interpret this as the floor $\lfloor C^* \log n\rfloor$, but as rounding will not affect the estimates below, we omit the floor from our notation.  Any vertices left over after dividing $n$ by $3 C^* \log n$ will not be part of any region, and these extra vertices will number at most $3 C^* \log n$. We will further divide the region $i$ into left, center, and right subregions consisting of $C^* \log n$ vertices each, so the left subregion of region $i$ contains the vertices $\{3 C^* (i-1) \log n, \ldots ,3 (i-1) C^* \log n + C^* \log n - 1\}$, the center subregion contains vertices $\{3 C^* (i-1) \log n + C^* \log n, \ldots , 3 C^* (i-1) \log n + 2 C^* \log n - 1\}$, and the right subregion contains vertices $\{3 C^* (i-1) \log n + 2 C^* \log n, \ldots, 3i C^* \log n - 1\}$.

Let the spacetime region $R_i$ consist of the vertices and edges of the $i$th region through time $C \log n$, where $C$ will be chosen later depending on $\alpha$, $\lambda$, {and $C^*$}. We will call region $R_i$ \textbf{broken} if it contains no infected vertices in the center subregion at time $C \log n$. 

Our first step is to show that starting from all vertices infected, the indicators of the events that the $i$th region is broken for $1\le i < n/ (3C^*\log n)$ stochastically dominate a collection of independent Bernoulli($p$) random variables for some $p>0$ (not depending on $n$). To this end, we let $\mathcal{G}_i$ denote the $\sigma$-field generated by the graphical construction (times of infection arrows, recovery dots, and avoidance crosses) \textit{outside} of $R_i$ during the time interval $[0, C\log n]$. We will show that $\mathbb{P}(R_i \text{ is broken}\ |\ \mathcal{G}_i) \ge p > 0$ for every $i$. In what follows we refer to the \textbf{graphical construction restricted to $R_i$}, by which we mean the graphical construction with only those symbols (infection arrows, recovery dots, and avoidance crosses) that remain after erasing all symbols outside of $R_i$. 

We define the \textbf{gap-edge process} (analogous to the process described in Section~\ref{sec:Z lower bd}) in $R_i$ as follows. Suppose the middle vertex of $R_i$ (a designated vertex closest to the midpoint of the region) is healthy at time $1$. Define the left gap-edge process, $l_g(t)\le 0$, to track the displacement from the middle vertex to its closest infected neighbor to the left in $R_i$ at time $t \ge 1$, with the conventions that $l_g(t) = -\infty$ if there are no such infected neighbors in $R_i$ and $l_g(t)$ ``hits'' $0$ if the middle vertex of $R_i$ is reinfected from the left. The states $-\infty$ and $0$ act as absorbing states for $l_g(t)$. Likewise, define the right gap-edge process, $r_g(t)$, to track the displacement from the middle vertex to its closest infected neighbor to the right in $R_i$ at time $t \ge 1$, with the convention that $r_g(t) = \infty$ if no such vertex exists.

We next define the following events.

\begin{enumerate}
\item The middle vertex of the center subregion of $R_i$ is healthy at time $1$. Call this event $\mathcal{A}^i_1$.
\item If $\mathcal{A}^i_1$ occurs, at time $1$, consider the left and right gap-edge processes around the middle vertex during time $[1, C \log n]$. These edge processes both leave the region $R_i$ before infecting the middle vertex (hitting $0$) and before time $C \log n$. Call this event $\mathcal{A}^i_2$.

\item Let $0\le t_1<t_2<\cdots< t_k \le C\log n$ be all the times at which there are infection arrows along the leftmost edge in $R_i$, $(3 C^* (i-1) \log n - 1, 3 C^* (i-1) \log n)$, up to time $C\log n$. We say that an uninterrupted path of infection exists in the left subregion of $R_i$ during time $[0,C \log n]$ if there exists $j\le k$ such that, starting at time $t_j$ with all edges active and all vertices healthy except for a single infection at $3 C^* (i-1) \log n$, in the graphical construction restricted to $R_i$ the vertex $3C^*(i-1)\log n + C^*\log n$ gets infected by time $C\log n$. Call the complementary event (that no uninterrupted path of infection exists) $\mathcal{A}^i_3$.

\item Analogously define an uninterrupted path of infection in the right subregion, and let $\mathcal{A}^i_4$ be the event that no such path exists.
\end{enumerate}

If $\mathcal{A}^i_1\cap \mathcal{A}^i_2$ occurs, the left and right gap-edge processes will leave region $R_i$ before reinfecting the middle vertex at random times $\tau_l, \tau_r \leq C \log n$. If $\tau_l\le \tau_r$, then at time $\tau_l$ all vertices are healthy and all edges (with possible exception of the leftmost edge) are active to the left of the middle vertex in $R_i$. Therefore, if $\mathcal{A}^i_3$ also occurs, the vertex $3C^*(i-1)\log n + C^*\log n$ cannot be reinfected during time interval $[\tau_l,\tau_r]$, and at time $\tau_r$ all vertices in the middle and right subregions of $R_i$ are healthy and all edges (except possibly the rightmost edge) are active. Therefore, if $\mathcal{A}^i_4$ also occurs, then no vertices in the middle subregion of $R_i$ will be infected at time $C\log n$. The case $\tau_l>\tau_r$ is similar, and we have
\begin{equation}\label{eq:A implies broken}
\cap_{j=1}^4 \mathcal{A}^i_j \subset \{R_i\text{ is broken}\}.
\end{equation}
We next estimate the conditional probability of the event $\cap_{j=1}^4 \mathcal{A}^i_j$ given $\mathcal{G}_i$.

For $\mathcal{A}_1^i$ to occur, it is sufficient for the middle vertex to recover at some time in $[0,1]$ and have no incoming infection arrows during $[0,1]$. These events are independent of $\mathcal{G}_i$, so
\begin{equation}\label{eq:A_1}
\mathbb{P}(\mathcal{A}_1^i\ |\ \mathcal{G}_i) \ge e^{-2\lambda}(1-e^{-1}).
\end{equation}

Next we bound from below the probability of $\mathcal{A}_2^i$ given $\mathcal{G}_i\cap \mathcal{A}_1^i$. On the event $\mathcal{A}_1^i$, consider the left gap-edge process, $l_g(t)$. Observe that $l_g(t)$ increases (by 1) only if the rightmost infected vertex to the left of the middle infects its neighbor to the right before either recovering or being avoided by its neighbor to the right (deactivation of the edge); otherwise $l_g(t)$ will decrease by at least 1. Note that when $l_g(t)$ jumps to the left, the size of the jump may depend on $\mathcal{G}_i$, but until $l_g(t)\in \{-\infty,0\}$ and conditional on $\mathcal{G}_i$, we have that $l_g(t)$ takes steps to the right with probability at most $\frac{\lambda}{\lambda+1+\alpha}<\frac{1}{2}$ and otherwise steps left. Therefore, the sequence of locations of $l_g(t)$ after successive jumps (the embedded discrete-time ``chain'', which is not Markov) is stochastically dominated by a simple random walk that steps left with probability 
$$
\xi = \frac{1+\alpha}{1+\alpha+\lambda}>1/2,
$$
right with probability $1-\xi$, and starts at $-1$. (One can explicitly couple these stochastic processes, conditional on $\mathcal{G}_i$, until $l_g(t)$ hits either $-\infty$ or $0$, after which the random walk process can be extended independently for all time.)

By a standard Gambler's Ruin analysis, the dominating random walk process never returns to $0$ with probability $1-\frac{\lambda}{1+\alpha}>0$. Moreover, the location of the random walk after $K=\frac{3}{2 \xi - 1}C^*\log n$ steps is dominated by $K-2X$ where $X\sim \text{Bin}(K,\xi)$. Since $E(K-2X) = K(1-2\xi) = -3C^*\log n$ and $\text{Var}(K-2X) = 4K\xi(1-\xi) \le \frac{3}{2\xi-1}C^* \log n$, Chebychev's inequality implies that the probability that the random walk has not crossed $-2C^*\log n$ after $K$ steps is at most $\frac{3}{(2\xi - 1)C^* \log n}$. Next, observe that $l_g(t)$ makes jumps at least as frequently as the arrivals of recovery dots, and by independently generating recovery dots at rate $1$ after $l_g(t)$ is absorbed (to emulate additional jumps), the number of jumps made by $l_g(t)$ by time $C\log n$ stochastically dominates Poisson$(C\log n)$. Therefore, letting 
$$
C =\frac{6}{2\xi-1}C^*,
$$
the number of jumps made by time $C\log n$ exceeds $K$ with probability at least $1-\frac{2(2\xi-1)}{3C^*\log n}$ by Chebychev's inequality. Thus, $l_g(t)$ hits $-\infty$ by time $C\log n$ with probability at least $1 -  \frac{\lambda}{1+\alpha} - \frac{3}{(2\xi - 1)C^* \log n} - \frac{2(2\xi-1)}{3C^*\log n}$. An analogous argument shows $r_g(t)$ hits $\infty$ by time $C\log n$ with at least the same probability, and does so independent (conditional on $\mathcal{G}_i$) of $l_g(t)$. We conclude that for large $n$,
\begin{equation}\label{eq:A_2}
\mathbb{P}(A_2^i | \mathcal{G}_i\cap A_1^i) \ge \left(1 -  \frac{\lambda}{2(1+\alpha)}\right)^2.
\end{equation}

To see that $\mathcal{A}_3^i$ occurs with high probability first observe that $k$, the number of  infection arrows along $(3 C^* (i-1) \log n - 1, 3 C^* (i-1) \log n)$, which is the leftmost edge into $R_i$, during time interval $[0,C\log n]$ is Poisson($\lambda C\log n$) distributed, and therefore $k\le 2\lambda C \log n$ with probability at least $1-\frac{1}{\lambda C\log n}$. For each $1\le j\le 2\lambda C \log n$, at the time $t_j$ of the $j$th infection arrow along the leftmost edge into $R_i$, we begin tracking an infection process in the graphical construction restricted to $R_i$ (so this process is independent of $\mathcal{G}_i$) and with only $3 C^* (i-1) \log n$, which is the leftmost vertex in $R_i$, infected at time $t_j$. (Note that for different $j$, these processes may ``overlap" in their use of the symbols in the graphical construction in $R_i$, but this dependence will not matter. Also note that if $t_j>C\log n$, then this process will do nothing, as we are ignoring symbols outside of the space-time region $R_i$.)

Consider the $j$th such process, started at time $t_j$. As in the previous argument, the sequence of locations of the rightmost infected vertex in $R_i$ after each jump in its location is stochastically dominated by a simple random walk that moves left by $1$ with probability $\xi>1/2$ and right by $1$ with probability $1-\xi$. By the Gambler's Ruin, it follows that this random walk, when started from $3 C^* (i-1) \log n$, hits $3 C^*(i-1)\log n+C^*\log n$ before hitting $3 C^* (i-1) \log n - 1$ with probability at most 
$$
\left(\frac{\lambda}{1+\alpha}\right)^{C^*\log n} = n^{-C^* \log((1+\alpha)/\lambda)}.
$$
If the random walk hits $3 C^* (i-1) \log n - 1$ before hitting $3 C^* (i-1) \log n+C^*\log n$, then the infection processes started at time $t_j$ never reaches the middle subregion of $R_i$. Therefore, the probability that any of the first $2\lambda C\log n$ such infection attempts ever reaches the middle subregion of $R_i$ is at most
$$
2\lambda C(\log n) n^{-C^* \log((1+\alpha)/\lambda)},
$$
so for large $n$,
\begin{equation}\label{eq:A_3}
\mathbb{P}(\mathcal{A}_3^i\ |\ \mathcal{G}_i) \ge 1- \frac{2}{\lambda C\log n}.
\end{equation}
The same (symmetric) argument implies
\begin{equation}\label{eq:A_4}
\mathbb{P}(\mathcal{A}_4^i\ |\ \mathcal{G}_i) \ge 1- \frac{2}{\lambda C\log n}.
\end{equation}
Combining equations~\eqref{eq:A implies broken}--\eqref{eq:A_4}, we have for all large $n$,
\begin{equation}\label{eq:broken bd}
\mathbb{P}(R_i \text{ is broken}\ |\ \mathcal{G}_i) \ge \frac{1}{2}e^{-2\lambda}(1-e^{-1})\left(1 -  \frac{\lambda}{2(1+\alpha)}\right)^2 =: p.
\end{equation}

Now, let $i_1< \cdots < i_M$ be all of the random indices such that $R_{i_\ell}$ is broken for $1\le\ell\le M$, and let $v_1, \ldots, v_M$ be the middle vertices (in the middle subregions) of each $R_{i_\ell}$. By \eqref{eq:broken bd}, the collection of (indicators of) broken regions dominates a collection of independent Bernoulli$(p)$ random variables, of which there are fewer than $n$. The longest run of $0$'s in fewer than $n$ independent Bernoulli$(p)$ trials arranged in a cycle is smaller than $\frac{4}{\log((1-p)^{-1})}\log n$ with probability at least $1-n^{-1}$. So with probability at least $1-n^{-1}$ we have
\begin{equation}\label{gap_size}
\text{dist}(v_1,v_M) \vee \max_{\ell} \text{dist}(v_{\ell}, v_{\ell+1}) \le \frac{5}{\log((1-p)^{-1})}\log n \cdot (3 C^* \log n),
\end{equation}
where $\text{dist}$ is the shortest path distance on $\bZ_n$, and the $5$ in the numerator accounts for vertices that are not in any region (between $v_M$ and $v_1$).

We now sketch an argument, which is very similar to the arguments above, to show that for a large enough $C'>0$, during the time interval $[C\log n, C\log n + C'(\log n)^2]$, each interval of vertices between $v_\ell$ and  $v_{\ell+1}$ will fully recover without ever interacting with neighboring intervals. This will complete the proof of the subcritical result.

For $\ell = 1,\ldots M-1$, in the interval between $v_\ell$ and $v_{\ell+1}$ (and between $v_M$ and $v_1$) there are no infected vertices within $\frac12 C^*\log n$ of $v_\ell$ or $v_{\ell+1}$ at time $C\log n$. Starting at time $C\log n$,  the sequence of distances from $v_\ell$ to the leftmost infected vertex in $[v_\ell, v_{\ell+1}]$ after each jump in its location dominates a random walk that moves to the right with probability $\xi>1/2$ (with the convention that the distance is $\infty$ if there are no infected vertices in the interval, and we ignore infections coming from outside the interval, but as we will see, there are none). Choosing
$$
C^* = \frac{4}{\log((1+\alpha)/\lambda)},
$$
it follows that the probability that the leftmost infected vertex in this interval ever reaches $v_\ell$ is at most
$$
\left(\frac{\lambda}{1+\alpha}\right)^{(C^*/2)\log n} = n^{-2}.
$$
Likewise, the probability that the rightmost infected vertex in this interval ever reaches $v_{\ell+1}$ is at most $n^{-2}$, so the probability that there exists an $\ell$ such that $v_\ell$ is ever reinfected is at most $2n^{-1}$ (and on the complementary event, we are justified in ignoring potential infections between neighboring intervals).

The leftmost infected vertex between $v_\ell$ and $v_{\ell+1}$ attempts to jump at rate at least 1 (extending the `jump' process, as before, beyond the first time that either there are no infected vertices in the interval or $v_\ell$ gets infected), so the number of jumps during the time interval $[C\log n,C\log n + C'(\log n)^2]$ dominates a Poisson$(C'(\log n)^2)$ distribution. Therefore, by a standard lower tail estimate for the Poisson distribution, the probability that the number of attempted (potential) jumps by the leftmost infected vertex is less than $\frac12 C'(\log n)^2$ is at most $e^{-(\log n^2)} \le n^{-2}$ for $C'$ and $n$ sufficiently large. After $K = \frac12 C'(\log n)^2$ jumps, the displacement of a simple random walk that moves right with probability $\xi$ is $2Y-K$ where $Y\sim $ Bin$(K, \xi)$, and by a Chernoff bound, $2Y-K$ exceeds $\frac12(2\xi -1)K = \frac14(2\xi-1) C'(\log n)^2$ with probability at least $1-n^{-2}$. The displacement of the leftmost infected vertex after $K$ jumps exceeds this, so on the event in \eqref{gap_size}, if $C'\ge \frac{60C^*}{(2\xi-1)\log((1-p)^{-1})}$, then the leftmost infected vertex will exceed $v_{\ell+1}$, which implies the interval between $v_\ell$ and $v_{\ell+1}$ is fully recovered. Finally, by a union bound, the probability that one of the $M$ intervals contains an infected vertex at time $C\log n + C'(\log n)^2$ is at most $Mn^{-2}\to 0$ as $n\to\infty$. \qed

\subsection{Supercritical proof (large $\lambda$)}
To prove the upper bound, we will use our result from section 3 that the contact process with avoidance can stochastically dominate an oriented site percolation with probability of occupancy $p$ for any chosen $p < 1$ so long as $\lambda$ is chosen to be sufficiently large, along with some facts about oriented percolation. We begin by briefly describing the models, introducing some notation, and stating some results about oriented percolation that we need.

Oriented percolation is defined on the sites $\{(x,t) \in \boldZ \times \mathbb{N}: x = (t \mod 2) \mod 2\}$ where $(x,t)$ and $(y,s)$ are neighbors when $|x-y| = 1$ and $|t-s| = 1$. x can be thought of as space and t as time. In site percolation each site is either occupied with probability $p_s$ or unoccupied with  probability $1-p_s$ independently, and two sites are connected if they are neighbors and both are occupied. In bond percolation, bonds between neighboring sites are active with probability $p_b$ and inactive with probability $1-p_b$ independently, and two sites are connected if there is an active bond between them. Define $S_n^A = \{(x,t); t=n \textrm{ and there is a connected path from some }(y,s) \in A \textrm{ to } (x,t)\}$ for site percolation and $B_n^A$ analogously for bond percolation.  We write $\{A \rightarrow \infty\}$ to mean that there is an infinite oriented path starting from the set $A$. We denote the critical value for site percolation $p_s(c)$ where $p_s(c)$ is the unique value such that $\prob{\{A \rightarrow \infty \}} > 0$ if and only if $p_s > p_s(c)$ for site percolation and define $p_b(c)$ in the same way for bond percolation. We now state some results.

\begin{prop}
For any $A$, $p_b$, for all $p_s \geq p_b(2-p_b)$, $B_n^A \underset{stoch.}{\subset} S_n^A$ for every $n$. 
\end{prop}

This can be easily shown by a straightforward coupling argument. See Liggett (1999) \cite{Liggett1999} for details.

\begin{prop}[\cite{percolation}]
For any set $D$, $\prob{\{D \rightarrow \infty\}^c} \leq Ce^{-\gamma |D|}$ for some constants $C, \gamma > 0$.
\end{prop}
In words, the probability that an oriented site percolation dies is exponentially small in the size of the starting set.


\begin{prop}[Tzioufas 2014 \cite{Tzioufas}]
Suppose $p_b > p_b(c)$. Then for any $p^* < p_b$ and any finite set $D$ of consecutive sites at time $n$, $\prob{|B_n^{2\boldZ} \cap D| < p^*|D|} \leq Ce^{-\gamma |D|}$ for some constants $C, \gamma > 0$.
\end{prop}
This result follows from Theorem 1 of \cite{DS1988}. The result is stated for bond percolation but using proposition 4.1 we can also apply it to site percolation.

The previous result concerns oriented percolation on the infinite lattice $2\boldZ$ where each site or bond initially has some probability $p$ of being occupied/active. However, in our comparison percolation the active sites are determined at the start of each cycle. We can remedy this technical difficulty by showing that for appropriately chosen $D$ and any $k \in \Zn$ we have with high probability $S_n^{2\boldZ} \cap D = S_n^{\{k\}}$ conditional on the event $F_k = \{\{k\} \rightarrow \infty\}$. To that end we require the following result.

\begin{prop}[Durrett 1984 \cite{percolation}]
Define the right edge $r_n = \sup_x\{(x,t):\textrm{x is occupied and }t = n\}$ of a supercritical oriented site percolation starting from $\{k\}$ such that $F_k$ occurs. Then there exists $a = a(p_s) \in (0,1)$ such that $\prob{r_n \leq k+an} \leq Ce^{-\gamma n}$ for some constants $C, \gamma > 0$
\end{prop}
By symmetry an analogous result holds for the left edge $l_n$. Also note that oriented percolation is translation invariant so without loss of generality we can take $k = 0$.

Suppose $p > p_b(c)(2 - p_b(c))$. Then if we take the set $D$ in proposition 4.5 to be $\{-an, \ldots ,an\}$ then 
$$\prob{|S_n^{2\boldZ} \cap D| \leq p^*2an | F_0} \leq Ce^{-\gamma 2an}.$$
Now note that $\prob{r_n \leq an | F_0} \leq Ce^{-\gamma n}$ and $\prob{l_n \geq -an | F_0} \leq Ce^{-\gamma n}$, so conditional on $F_0$ with probability $1 - Ce^{-\gamma n}$ any $0 \neq x \in S_0^{2\boldZ}$ for which there is a path from $x$ to some $y \in D$ must intersect one of the edges of the percolation starting from $\{0\}$ in which case we have 
$$S_n^{2\boldZ} \cap D|F_0 = S_n^{\{0\}} \cap D|F_0$$
and so 
$$\prob{|S_n^{\{0\}} \cap D| \leq p^*2an|F_0} \leq Ce^{-\gamma 2an}$$

We are now ready to commence the proof of the exponential survival regime. Note that while the values of the constants $C$ and $\gamma$ change from line to line, the values themselves are uninteresting, and we only define finitely many different constants and so can take minima and maxima as needed.

\begin{proof}
Fix $\al > 0$, choose $p^* > p_b(c)$ and $p > p^*(2-p^*)$. By section 3 we can choose $\lam$ large so that the contact process stochastically dominates an oriented site percolation with occupancy probability $p$, so we consider that process on $\Zn$ instead, starting from all sites occupied.
Divide $\Zn$ into two halves, say $\{0, \ldots n/2 - 1\}$ and $\{n/2 \ldots n\}$. If we ignore for the moment the second half and consider only the first half, we can note that the by proposition 4.3, a percolation starting from the first half survives forever on $2\boldZ$ (and thus also until time $n$) with probability at least $1 - Ce^{\gamma n/2}$ since the starting set has size $n/2$. Using translation invariance we can conclude from this that $\prob{F_0} \geq 1 - Ce^{\gamma n/2}$.

Now suppose the percolation from the first half survives until time $n$ (which it does if $F_0$ occurs) and thus has had an opportunity to spread across the second half, but not to wrap back on itself. Choose $\epsilon < p/3$. Then by propositions 4.2, 4.4, and 4.5 
$$\prob{|A_{n/2}^{\{0, \ldots n/2 - 1\}}| < 2\epsilon n} \leq \prob{|A_{n/2}^{\{0, \ldots n/2 - 1\}}| < 2\epsilon n|F_0} + \prob{F_0^c} \leq Ce^{\gamma n}.$$
If we then have at least $2\epsilon n$ occupied sites on $\Zn$, we must be able to take one half of $\Zn$ that has at least $\epsilon n$ occupied sites, and we can repeat the process, again starting from a set with size $O(n)$. Call each time through these steps a cycle. By a union bound on the probability of failure at each step, the probability of a successful cycle is at least $1 - Ce^{-\gamma n}$. Thus, $\tau$ stochastically dominates a geometric random variable with success probability $1 - Ce^{-\gamma n}$, and so there exist $C, \gamma > 0$ such that $\prob{\tau \leq Ce^{\gamma n}} \rightarrow 0$ as $n \rightarrow \infty$.
\end{proof}

\section{Results for the star graph}
The goal of this section is to prove Theorem~\ref{star-thm}. On the star graph it is possible to reformulate the CPA model by only assigning states to the center and the leaves, and not individual edges. We do this as follows. The center takes on values in $\{0,1\}$ meaning healthy and infected as before. Leaves take on states in $\{0,1\} \times \{A,D\}$ where 0 and 1 denote healthy and infected (vertex) states, and $A$ and $D$ denote active and inactive (edge) states. Active leaves can both receive and transmit the infection, while inactive leaves can do neither. Depending on the state of the center, the system follows different dynamics.

\begin{definition}\label{stardef}
When the center is infected (one-phase)
\begin{enumerate}
\item $0A \rightarrow 1A$ at rate $\lam$ (Center infects leaf)
\item $1A \rightarrow 0A$ at rate $1$. (Leaf recovers)
\item $0A \rightarrow 0D$ at rate $\al$ (Leaf avoids center)
\item $1D \rightarrow 0A$ at rate $1$. (Leaf that had been avoided by the center in a previous zero-phase recovers)
\item The center goes from $1 \rightarrow 0$ at rate $1$.
\end{enumerate}

When the center is healthy (zero-phase)
\begin{enumerate}
\item At the time of the center's recovery set all $0D$ to $0A$ (Leaves stop avoiding the center)
\item $1A \rightarrow 0A$ at rate $1$. (Leaf recovers)
\item $1A \rightarrow 1D$ at rate $\al$ (Center avoids leaf)
\item The center goes from $0 \rightarrow 1$ at rate $m\lam$ where $m$ is the current number of $1A$ leaves.
\end{enumerate}
\end{definition}

The system is perhaps most easily understood by referring to Figure~\ref{fig:star}. Some thought reveals that on the star graph this is equivalent to the formulation of the CPA model given in the introduction.

\begin{figure}
\includegraphics[scale=.33]{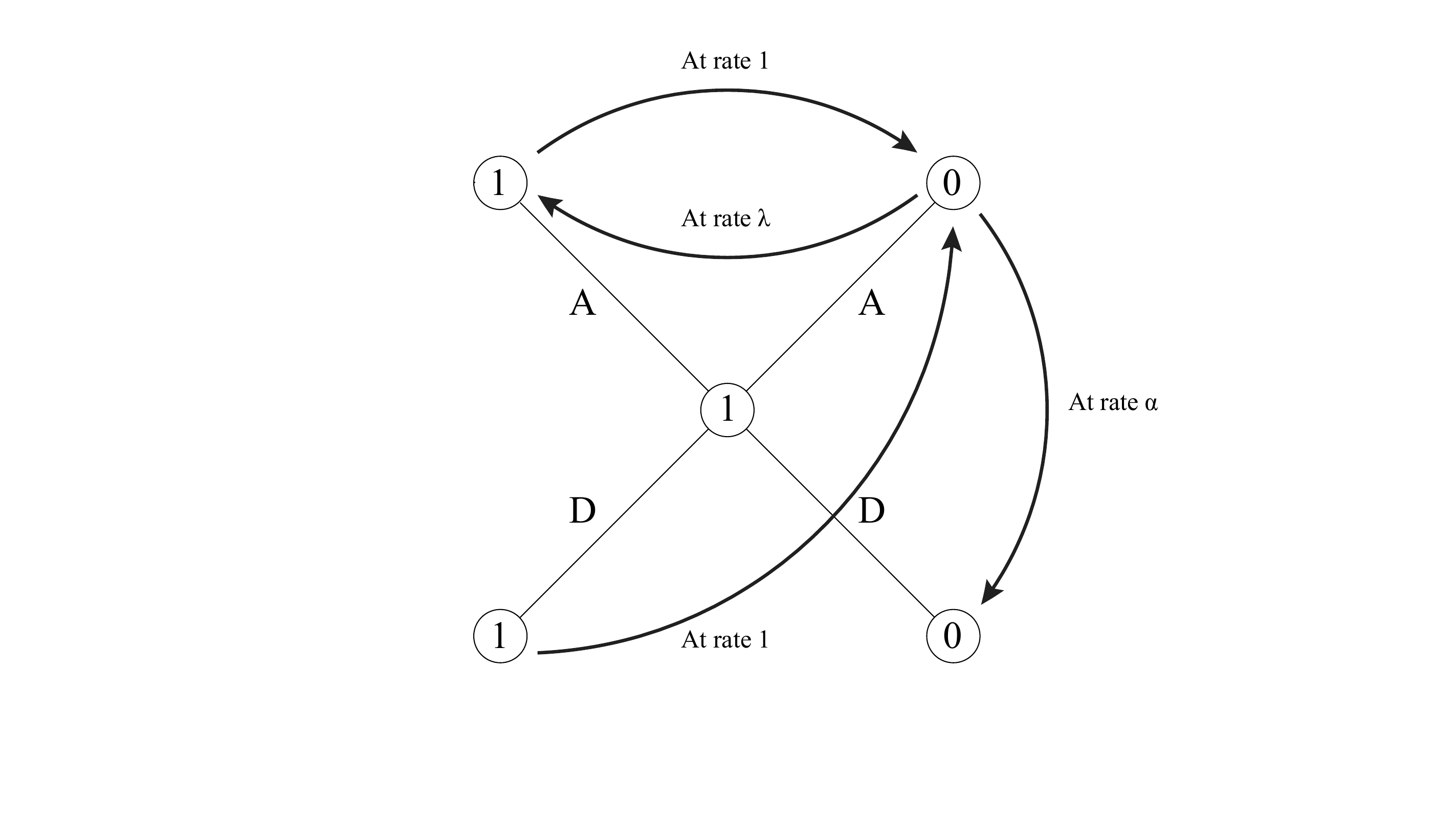}
\includegraphics[scale=.33]{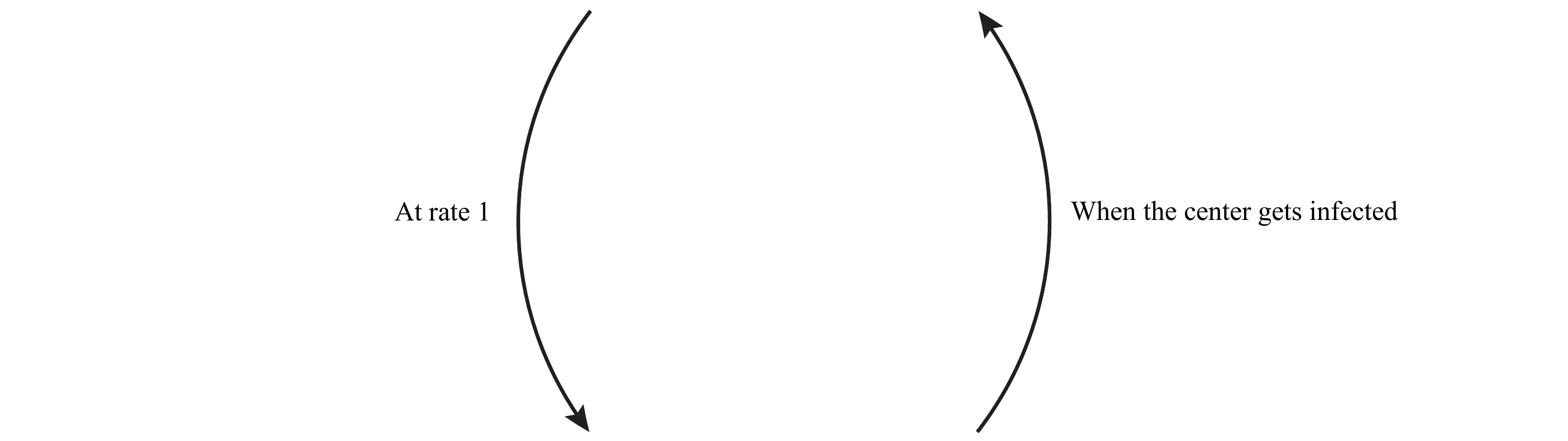}
\includegraphics[scale=.33]{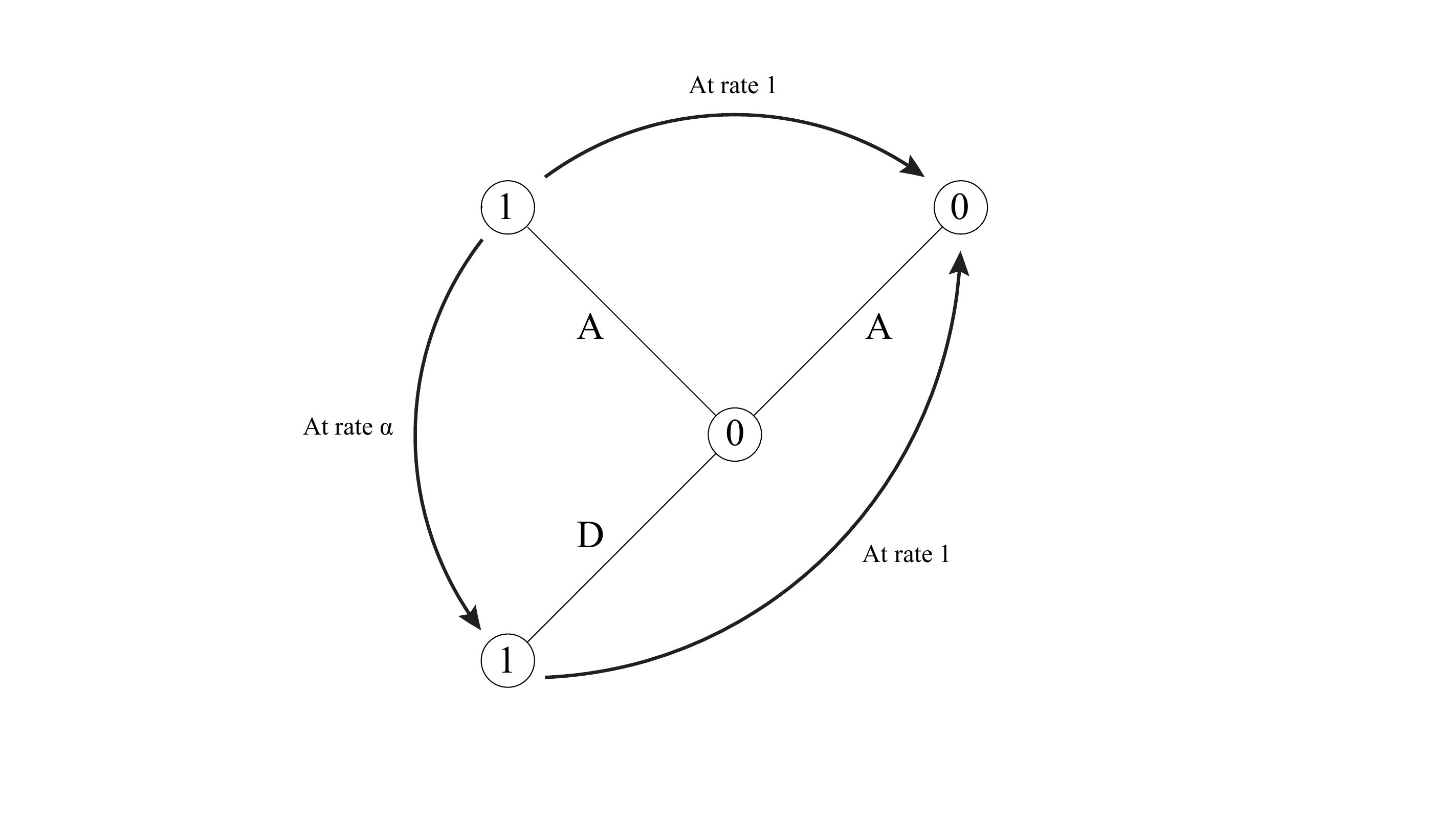}
\caption{\label{fig:star} Star graph dynamics}
\end{figure}

The center plays a special role in the dynamics of the star graph, and so to understand the process dynamics on the star graph we consider in turn the dynamics when the center is infected (which we will call the one-phase,) the dynamics when the center is healthy (which we will call the zero-phase,) and change in the number of infected leaves between consecutive one-phases.

When the center is infected, the set of possible states of the other nodes is $S=\{1A, 0A, 0D, 1D\}$. If $X_t$ denotes the state of a typical node at time $t$ and if 
\[ V_{i,j}(t) =\prob{X_t=j| X_0=i} \text{ for } i, j \in S \text{ and }\vV=((V_{i,j}))_{i,j \in S},\]
then using standard arguments for continuous time Markov chain it is easy to see that 
\[ \vV'(t)=\vV(t)\vA, \text{ where } \vA=\begin{array}{c}
1A \\0A \\0D \\1D \end{array}
\left[ \begin{array}{cccc}
-1 & 1 & 0 & 0 \\
\lambda & -(\lambda+\alpha) & \alpha & 0 \\
0 & 0 & 0 & 0\\
0 & 1 & 0 & -1
\end{array}
\right]
\]
In order to find the eigenvalues of $\vA$ note that 
\[det(\vA-\gamma I) = \gamma(-1-\gamma)(\al \gamma + \al + \lam \gamma + \gamma^2 + \gamma)\]
So the eigenvalues are $0,  -\gamma_1, -\gamma_2$ and -1, where 
\[ \gamma_1=\frac 12[(\lambda+\alpha+1)-\sqrt{(\lambda+\alpha+1)^2-4\alpha}],
\gamma_2=\frac 12[(\lambda+\alpha+1)+\sqrt{(\lambda+\alpha+1)^2-4\alpha}].\]
Simple algebra shows that
\beqa \label{gammaprop}
\gc_1+\gc_2=\gl+\ga+1, \quad \gc_1\gc_2 = \ga, \quad \gc_2-\gc_1=[\gl^2+(\ga-1)^2+2\gl(\ga-1)]^{1/2}, \\
0 \le \gamma_1\le 1\le \gamma_2 \le 1+\gl + \al.
\eeqa
We note that this differs from the case of the classical contact process, in which the matrix $\vA$ has rank 2 and eigenvalues $0,0,-1,$ and $-(1+\lambda)$.

From the description of the process and the generator matrix, we can make a heuristic argument for why the survival time on the star is like $n^{\Delta}$ where $\Delta = 1/\gamma_1$. Essentially, the process dies when we observe a long one-phase during which all the leaves start avoiding the center. Suppose we start at the beginning of a one-phase. Then this one-phase lasts for time T where $T \sim$ Exp$(1)$. At time T, we expect the number of $0D$ leaves to be about $ne^{-\gamma_1 T}$, which is less than 1 when $T > (1/\gamma_1)\log n$. Since T is an Exp$(1)$ random variable, $\prob{T > (1/\gamma_1)\log n} = n^{-1/\gamma_1}$ so we should need about $n^{-1/\gamma_1}$ one-phases to observe a one-phase long enough for the process to die, and each one-phase-zero-phase cycle lasts on average for $O(constant)$ time, suggesting a survival time of $n^{\Delta}$. Of course, there are many details that need to be verified, and the remainder of this section is devoted to making this heuristic argument rigorous. 

The right eigenvectors for those eigenvalues are the columns of the following matrix.
\[ \vB=\left[\begin{array}{cccc}
               1 & \zeta & \zeta & 0 \\
               1 & 1-\gamma_1 & 1-\gamma_2 & 0 \\
               1 & 0 & 0 & 0 \\
               1  & 1 & 1 & 1
               \end{array}\right].
 \]                                       
   where $\zeta = \frac{(1-\gamma_1)(\gamma_2-1)}{\lambda}$\\ So $\vA\vB=\vB\vD$, where $\vD=Diag(0, -\gamma_1, -\gamma_2, -1)$. Let $\vW(t)=\vV(t)\vB$, so
   \[ \vW'(t) = \vV'(t)\vB = \vV(t)\vA\vB = \vV(t)\vB\vD = \vW(t)\vD.\]
   Since $\vD$ is diagonal, we have $\vW(t) = \vW(0)\exp(\vD t)$. Hence, using the fact that $\vV(0)=I$
   \[ \vV(t) = \vB\exp(\vD t) \vB^{-1}.\]
 
   \begin{lemma}\label{uvprop}
   If $u(t)  :=  \vV_{1A,1A}(t)$ and  $v(t)  :=  \vV_{0A,1A}(t)$, then
   \begin{enumerate}
 \item $u(\cdot)$ is decreasing,  $u(0)=1$ and $u(t) \downarrow 0$  as  $t\to\infty$ exponentially fast in $t$.,
  \item $v(0)=0$, $v(\cdot)$ is increasing (resp.~decreasing) for $t\le$ (resp. $\ge$)  $(\log\gamma_2-\log\gamma_1)/(\gamma_2-\gamma_1)$ and $v(t) \to 0$ as  $t\to\infty$ exponentially fast in $t$,
 \item $v(t) \le  u(t)$ for all $t\ge 0$,
\item the map $\eta\mapsto f(\eta):=\int_0^\infty (\gh u(t)+(1-\gh)v(t))e^{-t}\; dt -\eta$ is monotonically decreasing, and $f(\eta) \ge 0$ (resp.~$\le 0$) for $\eta \le$ (resp.~$\ge$) $\gl/(\gl+\ga+2)$ 
 \end{enumerate}
   \end{lemma}
   
   \begin{proof}  By computation we see that
     \begin{align*}  
     u(t) & :=  \frac{(1-\gamma_1)e^{-\gamma_2t} - (1-\gamma_2)e^{-\gamma_1t}}{\gamma_2-\gamma_1} = \frac{\gamma_2-1}{\gamma_2-\gamma_1}e^{-\gamma_1t}+\frac{1-\gamma_1}{\gamma_2-\gc_1}e^{-\gamma_2 t}\\
    v(t) & :=  \left(\frac{(1-\gamma_1)(\gamma_2 - 1)}{\zeta (\gamma_2 - \gamma_1)}\right)(e^{-\gamma_1 t} - e^{-\gamma_2 t}) = \left(\frac{\lambda}{\gamma_2 - \gamma_1}\right)(e^{-\gamma_1 t} - e^{-\gamma_2 t}).
    \end{align*}
  1. From the properties of $\gamma_1$ and $\gamma_2$ in \eqref{gammaprop} it is clear that $\gc_1, \gc_2>0$ and the coefficients of $e^{-\gamma_1t}$ and $e^{-\gamma_2t}$ in $u(t)$ are both positive.  \\
  2. We observe (a) $0<\gc_1<\gc_2$, (b) $v(t)$ is a multiple of $e^{-\gc_1t} - e^{-\gc_2t}$, and (c) $v'(t)$ vanishes at $t=(\log\gc_2-\log\gc_1)/(\gc_2-\gc_1)$.\\
3. From the properties of $\gamma_1$ and $\gamma_2$ in \eqref{gammaprop}
\[v(t) \le \frac{\lambda}{\gamma_2-\gamma_1}e^{-\gamma_1t} \le \frac{\gamma_2-1}{\gamma_2-\gamma_1}e^{-\gamma_1t} \le u(t) \forall t\ge 0.\]
4.. Since $\int_0^\infty e^{-(1+a)t}\; dt = (1+a)^{-1}$ for any $a>0$,
\beqax
f(\gh) & = & \frac{1}{\gc_2-\gc_1}\left[\gh\left(\frac{\gc_2-1}{1+\gc_1}+\frac{1-\gc_1}{1+\gc_2}\right) + (1-\gh)\left(\frac{\gl}{1+\gc_1}-\frac{\gl}{1+\gc_2}\right)\right]-\eta \\
& = & \frac{1}{(1+\gc_1)(1+\gc_2)}\left[\gh(\gc_2+\gc_1) + (1-\gh)\gl\right]-\eta \\
f'(\gh) & = & \frac{\gc_1+\gc_2-\gl}{(1+\gc_1)(1+\gc_2)}-1.
\eeqax
Since $\gc_1+\gc_2=\gl+\ga+1$ and $\gc_1\gc_2=\ga$, it is easy to see that $f'(\gh)<0$ and
\[f(\eta)=0 \Leftrightarrow \eta=\frac{\gl}{(1+\gc_1)(1+\gc_2)-(\ga+1)} = \frac{\gl}{\gl+\ga+1}.\]
\end{proof}

Note that in the case of the classical contact process, 1. and 2. above do not hold, and $u(t)$ and $v(t)$ do not converge to $0$. It is the addition of avoidance allows the leaves to eventually avoid the center, after which they can no longer become infected during the current one-phase. This is the key difference driving the differing survival behaviors of the classical contact process and the CPA process on the star graph.

Next we will focus on the evolution of the number of nodes in different states of $S$. Note that at the beginning of each one-phase there is no node with state 0D and all nodes with state  0D at the end of each one-phase change their state to 0A at the beginning of the next zero-phase. Since the total number of nodes is $n$ (the size of the star graph), it suffices to keep track of the number of nodes in states 1A and 1D at the beginning of the one-phases. 
    
   When there are $m$ nodes in state 1A any time during a zero-phase, the rate at which the center gets infected is $\lambda m$. Also the rates at which nodes change their states to 0A and 1D are $m$ and $\alpha m$ respectively. Therefore, the time to the next event is exponentially distributed with mean $1/(\lambda+\alpha+1)m$, and the probability that the next event is the center becoming infected (before any nodes change their states) is
    \begin{equation}\label{lamhat}
     \hat\lambda=\frac{\lambda}{1+\alpha+\lambda},
     \end{equation}
    which does not depend on $m$.
    So if $N$ is the number of 1A nodes lost during a zero-phase, then $N$ has shifted Geometric  distribution with success probability $\hat\lambda$.
    \beq\label{Ndef}
     \prob{N=k} = (1-\hat\lambda)^k \hat\lambda, \quad k\ge 0.\eeq  
   Each of those nodes that changes its state from 1A becomes 1D or 0A with probability $\alpha/(1+\alpha)$ and $1/(1+\alpha)$ respectively. So the number of 1D nodes added during a zero-phase conditioned on $N$ is Binomial with parameters $N$ and $\alpha/(1+\alpha)$.  Unconditionally its distribution is shifted Geometric with success probability $\lambda/(\lambda+\alpha)$. Also if $\{T_i\}_{i\ge 0}$ denote the sequence of the above event times, then $T_0=0$ and $T_{i+1}-T_i$ has exponential distribution with mean $1/(\lambda+\alpha+1)(m-i)$. If a node changes its state from 1A to 1D at time $T_i$, then it stays 1D till the end of the current zero-phase with probability $\exp(-(T_{N+1}-T_i))$, where $N$ is the number of nodes lost from 1A.

   On the other hand, if $T$ is the duration of a phase 1, then during this phase a node with state 1D does not change its state with probability $e^{-T}$. Some of the 1D nodes that change to 0A nodes could then also change to 1A nodes before the end of the 1 state. If we let $\{\sigma_i\}_{i \geq 0}$ denote the times when 1D nodes change to 0A where $\sigma_0 = 0$, then the number of 1D nodes that change to 1A is $\sum_{i=1}^{L - \tilde L}Ber \left(v(T - \sigma_i)\right)$. Similarly using the $\{T_i\}_{i\ge 0}$ defined previously the number of 1A nodes that change to 1D and stay 1D until the end of a zero phase is $\sum_{i=1}^{N}\mathrm{Ber} \left(\frac{\al}{1+\al}(T_{N+1} - T_i) \right)$
    
    Using the above argument and the notation  $u(t)$ and $v(t)$ as in Lemma \ref{uvprop},
    we see that the transition in the number of 1A and 1D states at the beginning of two consecutive phase 1 states of the system can be described as follows.
    \begin{align*} 
    \left(\begin{array}{c}
                      K \\
                      L\end{array} \right)  
         & \xrightarrow[]{\text{ during phase 1}} 
                      \left(\begin{array}{c}
                      \tilde K = \mathrm{Bin}(K,u(T)) + \mathrm{Bin}(n-K-L,v(T)) + \sum_{i=1}^{L - \tilde L}\mathrm{Ber} \left(v(T-\sigma_i)\right)\\
                       \tilde L = \mathrm{Bin}(L, e^{-T}) 
                       \end{array}\right),\\
         \left(\begin{array}{c}
                      \tilde K \\
                      \tilde L\end{array} \right)  
         & \xrightarrow[]{\text{ during phase 0}} 
                      \left(\begin{array}{c}
                      \left( \tilde K - N \right)^+\\
                       \mathrm{Bin}(\tilde L, \exp(-T_{N+1})) + \sum_{i=1}^N \mathrm{Ber}\left(\frac{\alpha}{1+\alpha} (T_{N+1}-T_i)\right)
                       \end{array}\right),
    \end{align*}
    where $T \sim Exp(1)$, $N \sim Geom(\hat\lambda)$, $T_0=0$ and $(T_{i+1}-T_i) \sim Exp((\lambda+\alpha+1)(\tilde K-i))$. Conditionally on $T, N$, $\{\sigma_i\}_{i \geq 0}$ and $\{T_i\}_{i\ge 0}$ all Binomial and Bernoulli random variables are independent.
  
  In order to analyze the above Markov chain, let $(K_i, L_i)$ be the number of 1A and 1D nodes at the beginning of the $i$-th one-phase. We would like to be able to ignore the $1D$ nodes and analyze $K_i$ assume that $L_i = 0$ for all $i$. To do this we first need to introduce some new processes.
  
  If we assume transitions to $1D$ do not occur (so $L_i = 0$ for all $i$), then we obtain a new Markov chain $\{Z_i\}_{i=0}^\infty$, where $Z_i$ is the number of $1A$ nodes at the start of the $i^{\text{th}}$ one-phase. The sequence $\{Z_i\}_{i=0}^\infty$ is defined by $Z_0 = n$ and for $i\ge 0$,
\beq \label{Ztdef}
Z_{i+1} = (X_i + Y_i - N_i)^+,
\eeq
 where  
 \beqax
 &X_i \sim \mathrm{Bin}(Z_i, u(T_i)),  &Y_i \sim \mathrm{Bin}(n-Z_i, v(T_i)), \\
 &T_i \sim \mathrm{Exp}(1) \text{ and } &N_i \sim \text{shifted  Geometric}(\hat\lambda)  \text{ (as in \eqref{Ndef})},
 \eeqax 
where the coin flips involved in the Binomial expressions above are assumed to be conditionally independent of everything else given $Z_i$ and $T_i$, and $\{T_i: i\ge 0\}$ and $\{N_i: i\ge 0\}$ are assumed to be i.i.d.~sequences, independent of each other.

In order to justify studying the dynamics of $Z_i$ rather than $(K_i,L_i)$, we will show that there exists a good event $\sfG$ with probability going to $1$ as $n \rightarrow \infty$ on which there exists a coupling such that
\begin{equation}
Z_i^* \leq K_i \leq Z_i
\end{equation}
where $Z_i^* = Z_i - C^*\cdot(\log n)^3$ for all $i$ and $C^*>0$ is a constant, which will be chosen later to depend on $\alpha$ and $\lambda$.

In comparing $Z_i$ and $(K_i,L_i)$ we encounter two possible problems. First, we must establish that $1D$ vertices cannot accumulate in $(K_i,L_i)$. Second, even if the number of $1D$ vertices is bounded above, there may still be a significant flow of vertices from state $1A$ to $1D$ to $0A$, so we must establish that this drain of $1A$ vertices does not cause $Z_i$ and $K_i$ to drift too far apart. The good event $\sfG$ will ensure that neither of these happen. 

To define $\sfG$ we first define another process $\{W_i\}_{i=0}^{\infty}$. Let $C>0$ be a constant, and let $W_0 = n - C(\log n)^2$ and
\beq \label{Ztdef}
W_{i+1} = (X^W_i + Y^W_i - N^W_i)^+,
\eeq
 where  
 \begin{equation*}
 \begin{aligned}
 X^W_i &\sim \mathrm{Bin}(W_i-C(\log n)^2, u(T_i)),  \\
 Y^W_i &\sim \mathrm{Bin}(n-C(\log n)^2-(W_i-C(\log n)^2), v(T_i)), \\
 T_i &\sim \mathrm{Exp}(1), \\
 \text{ and } N^W_i &\sim \text{shifted  Geometric}(\hat\lambda)  \text{ (as in \eqref{Ndef})}.
 \end{aligned}
 \end{equation*}
In essence $\{W_i\}_{i=0}^{\infty}$ has $C(\log n)^2$ vertices removed that can be thought of as being fixed as $1D$ and at the beginning of each one-phase $C(\log n)^2$ vertices are converted directly from state $1A$ to $0A$. 
  
  \begin{lemma}\label{tight}
Let $\lambda, \alpha, \gamma > 0$. Then for any $\epsilon > 0$, there exist $N>0$ such that
\begin{enumerate}
\item $\prob{\max_{i = 1, \ldots n^{\gamma}}(L_i) > (12\gamma + 1)\frac{(\lambda + \alpha)(\gamma + 1)}{\alpha} (\log n)^2} < \epsilon$.
\item Let $C$ be the constant in the definition of the process $\{W_i\}$, and let $R_{i,k}$ be the number of the $C(\log n)^2$ vertices that were converted from $1A$ to $0A$ at the start of the $i$th one-phase that have not been reinfected by the start of the $k$th one-phase. Then 
$$
\prob{\max_{k=1 \ldots n^\gamma}\sum_{i=1}^k R_{i,k} \ge \Big(\frac{2\gamma}{\log(\frac{1 + \alpha + \lambda}{1 + \lambda})} + 1\Big) C  (\log n)^3} < \epsilon \text{ for all $n \geq N$.}
$$
\end{enumerate}
\end{lemma}  

\begin{proof}
In the process $(K_i,L_i)$, we first observe that if $J_i$ is the number of newly added $1D$ nodes during the $i$-th zero-phase, then $J_i$ is stochastically dominated by a geometric distribution on $\{0, 1, \ldots\}$ with success probability $\frac{\al}{\alpha+\lambda}$. Let $\sfA$ be the event $\{\max_{i=1, \ldots, n^{\gamma}} J_i \leq \frac{(\lambda+\alpha)(\gamma+1)}{\alpha} \log n\}$. Then 
\begin{equation}\label{Abd}
\prob{\sfA^c} \leq n^{\gamma}\frac{\lambda}{\lambda + \alpha}\left(\frac{1}{n}\right)^{\frac{(\lambda+\alpha)(\gamma+1)}{\alpha} \cdot \frac{\alpha}{\lambda+\alpha}} = o(1).
\end{equation}
For the process $W_i$, the number of $1A$ nodes newly converted to $0A$ at the start of the $i$th one-phase, $R_{i,i}$, is deterministically $C(\log n)^2$.

For the process $(K_i,L_i)$, if we ignore the conversion of $1D$ nodes to $0A$ nodes during zero-phases, then after the $k^\text{th}$ one-phase, the number of remaining $1D$ nodes from $J_i$ is $J_{i,k} \sim \text{Bin}(J_i, \exp[-\sum_{\ell=i+1}^k T_\ell])$, where $T_\ell$'s are iid Exp$(1)$ random variables. We now have 
\begin{equation}\label{L stoch J}
\sum_{i=1}^{k}J_{i,k} \sge L_k.
\end{equation}
In the case of $R_{i,k}$, a vertex that converted to $0A$ at the start of the $i$th one-phase becomes reinfected in a given future one-phase if the center attempts to infect it before either the center recovers or the vertex avoids the center. We can observe the probability that a vertex converted to $0A$ at the start of the $i$th one-phase becomes reinfected during the $l$th one-phase is $\frac{\lambda}{1 + \alpha + \lambda}$.

Observe $\sum_{\ell=i+1}^k T_\ell \sim \textrm{Gamma}(k-i,1)$, and so $\prob{\sum_{\ell=i+1}^k T_\ell < (k-i)/2} \leq e^{-(k-i)/6}$. Let $\sfB(k,k')$ be the event  $\left\{\sum_{\ell=i+1}^k T_\ell \ge \frac12(k-i)\text{ for all } 1\le i \leq k-k'\right\}$. Then for $0\le k'< k$, we have
\begin{equation}\label{Bkk'}
\prob{\sfB(k,k')^c} \leq (k-k')e^{-k'/6},
\end{equation}
and $\prob{\sfB(k,k')^c} = 0$ if $k'\ge k$.

If we let $X_i \sim \text{Bin}(\frac{(\lambda+\alpha)(\gamma+1)}{\alpha} \log(n), e^{-(k-i)/2})$, then 
\begin{equation}
J_{i,k} \mathbbm{1}_{\sfB(k,k')\cap \sfA}\sle X_i \textrm{ for all } i \leq k-k' \textrm{ and } k\leq n^{\gamma},
\end{equation}
Let $\sfD(k,k')$ be the event $\left\{J_{i,k} \mathbbm{1}_{\sfB(k,k')\cap \sfA} = 0 \text{ for all } 1\le i \leq k-k'\right\}$. Then 
\begin{equation}\label{Dkk'}
\begin{aligned}
\prob{\sfD(k,k')^c} &\leq \prob{X_i > 0 \text{ for some } 1\le i \le k-k'} \\
&\leq (k-k')(1 - (1 - e^{-k'/2})^{\frac{(\lambda + \alpha)(\gamma + 1)}{\alpha} \log n}).
\end{aligned}
\end{equation}

 Observe that for each $1\le k\le n^\gamma$, 
 \begin{equation}
 \{\sfA \cap \sfB(k,k') \cap \sfD(k,k')\}\subseteq \left\{\sum_{i = 1}^{k}J_{i,k} \leq \frac{(\lambda + \alpha)(\gamma + 1)}{\alpha} (\log n) k'\right\}.
 \end{equation}
 
To obtain part 1 of the Lemma, we choose $k' = (12 \gamma+1) \log n$,  so that the probabilities in \eqref{Abd}, \eqref{Bkk'} and \eqref{Dkk'} are sufficiently small:
\begin{equation}\label{choice of constants}
\begin{aligned}
& n^{\gamma}(\frac{1}{n})^{\frac{(\lambda + \alpha)(\gamma + 1)}{\alpha} \frac{\alpha}{\lambda+\alpha}} = o(1),\\
&n^{\gamma} e^{-(12\gamma+1) \log n/6} = o(n^{-\gamma}),\\
\text{and }\quad & n^{\gamma}(1 - (1 - e^{- (12 \gamma +1) \log n})^{\frac{(\lambda + \alpha)(\gamma + 1)}{\alpha} \log n}) = o(n^{-\gamma}).
\end{aligned}
\end{equation}
Then by \eqref{L stoch J} and \eqref{choice of constants},
\begin{equation}\label{max L_i bound}
\begin{aligned}
\prob{\max_{i = 1, \ldots n^{\gamma}}L_i > (12\gamma+1)\frac{(\lambda+\alpha)(\gamma+1)}{\alpha} (\log n)^2} &\leq \prob{\bigcup_{k=1}^{n^\gamma} \{\sfA \cap \sfB(k,k') \cap \sfD(k,k')\}^c}\\
& \leq \prob{\sfA^c} + \sum_{k=1}^{n^{\gamma}} \left[\prob{\sfB(k,k')^c} + \prob{\sfD(k,k')^c}\right]\\
& = o(1).
\end{aligned}
\end{equation}

We now obtain part 2 of the Lemma by an analogous argument. Let $\sfU(k,k')$ be the event $\{R_{i,k} = 0 \text{ for all } 1\le i \leq k-k'\}$. First, note that for any $1\le i \leq k-k'$ we have
\begin{equation}
\prob{R_{i,k} > 0} \leq C (\log n)^2\Big(\frac{1 + \alpha}{1 + \alpha + \lambda}\Big)^{k'}.
\end{equation}
From this we see
\begin{equation}\label{Ukk'}
\prob{\sfU(k,k')^c} \leq (k-k')C (\log n)^2\Big(\frac{1 + \alpha}{1 + \alpha + \lambda}\Big)^{k'},
\end{equation}
so we can choose $k' = (\frac{2\gamma}{\log(\frac{1 + \alpha + \lambda}{1 + \lambda})} + 1)\log n$ so that the probability in \eqref{Ukk'} is sufficiently small:
\begin{equation}\label{choice of constants 2}
\begin{aligned}
\sum_{k=1}^{n^\gamma}(k-k')C (\log n)^2\Big(\frac{1 + \alpha}{1 + \alpha + \lambda}\Big)^{k'} = o(1).
\end{aligned}
\end{equation}
Finally, we observe
$$
\prob{\max_{k=1 \ldots n^\gamma}\sum_{i=1}^k R_{i,k} \ge \Big(\frac{2\gamma}{\log(\frac{1 + \alpha + \lambda}{1 + \lambda})} + 1\Big)C (\log n)^3} \leq \sum_{k=1}^{n^\gamma}\prob{\sfU(k,k')^c} = o(1).$$
This completes the proof of the lemma.
\end{proof}

\begin{lemma}\label{couple}
Fix $\epsilon, \gamma > 0$, let $\sfG$ be the event 
$$
\left\{\sum_{i=1}^k R_{ik} \le C^*(\log n)^3 \hspace{2mm}\forall k \leq n^\gamma\right\} \cap \left\{\max_{i = 1, \ldots n^{\gamma}}(L_i) \le C (\log n)^2\right\}
$$ 
where $C = (12\gamma + 1)\frac{(\lambda + \alpha)(\gamma + 1)}{\alpha}$ in the definition of $\{W_i\}$ and $C^* =(\frac{2\gamma}{\log(\frac{1 + \alpha + \lambda}{1 + \lambda})} + 1)C$ in the definition of $\{Z_i^*\}$, and let $\tau^* = \inf\{i \geq 0: Z^*_i = 0\}$. Then for all sufficiently large $n$, we have $\prob{\sfG} > 1-\epsilon$. In addition
\begin{equation}
\begin{aligned}
& Z_i^*\mathbbm{1}_{\sfG} \sle K_i\mathbbm{1}_{\sfG} \qquad \text{for $0\le i\le \tau^*$, and}\\
& K_i\mathbbm{1}_{\sfG} \sle Z_i\mathbbm{1}_{\sfG} \qquad \text{for $0\le i\le \tau$.}
\end{aligned}
\end{equation}
\end{lemma}

\begin{proof}
Lemma \ref{tight} implies $\prob{\sfG}>1-\epsilon$. 

We first describe a coupling between $W_i\mathbbm{1}_{\sfG}$ and $Z_i\mathbbm{1}_{\sfG}$ that holds for $0 \leq i \leq \tau_W$. Begin by expanding the probability space in the usual way so that we can track the states of individual vertices. Next, note that $W_i\mathbbm{1}_{\sfG}$ and $Z_i\mathbbm{1}_{\sfG}$ are embedded discrete time processes of $Z_t\mathbbm{1}_{\sfG}$ and $W_t\mathbbm{1}_{\sfG}$ where $Z_t$ follows the process dynamics in Definition \ref{stardef} except that it ignores transitions to the $1D$ state and  $W_t$ follows the same process dynamics, also ignores transitions to the $1D$ state, and converts $C(\log n)^2$ randomly chosen $1A$ vertices to $0A$ vertices instantaneously at the start of each one-phase. We describe a coupling for $Z_t\mathbbm{1}_{\sfG}$ and $W_t\mathbbm{1}_{\sfG}$ during the one-phase as follows:
\begin{enumerate}
\item At the start of each one-phase, pair every $1A$ vertex in $W_t\mathbbm{1}_{\sfG}$ with a $1A$ vertex in $Z_t\mathbbm{1}_{\sfG}$ and as many $0A$ vertices in $Z_t\mathbbm{1}_{\sfG}$ with $0A$ vertices in $W_t\mathbbm{1}_{\sfG}$ as possible. Paired vertices share all random variables that determine their possible state changes, and unpaired vertices evolve independently according to their marginals.
\item During a one-phase, whenever an unpaired $0A$ vertex in $W_t\mathbbm{1}_{\sfG}$ becomes infected, pair it with an unpaired $1A$ vertex in $Z_t\mathbbm{1}_{\sfG}$. Whenever an unpaired $1$-state (infected) vertex in $Z_t\mathbbm{1}_{\sfG}$ recovers, pair it with an unpaired $0A$ vertex in $W_t\mathbbm{1}_{\sfG}$ if one exists.

\item During the zero-phase, instead couple the embedded discrete time processes $W_i\mathbbm{1}_{\sfG}$ and $Z_i\mathbbm{1}_{\sfG}$ by drawing a single $N_i$ to determine the number of $1A$ vertices that recover and distributing those recoveries uniformly at random among the available $1A$ vertices in each process.
\end{enumerate}

First observe that we have $W_t\mathbbm{1}_{\sfG} \leq Z_t\mathbbm{1}_{\sfG}$ which implies $W_i\mathbbm{1}_{\sfG} \leq Z_i\mathbbm{1}_{\sfG}$. Now note that in this coupling any vertices that are healthy in $W_i\mathbbm{1}_{\sfG}$ but infected in $Z_i\mathbbm{1}_{\sfG}$ must be vertices in the $W_i$ process that were converted from $1A$ to $0A$ at the start of a one-phase and have never since been reinfected. Thus when $\sfG$ occurs, $Z_i - W_i \leq C^*(\log n)^3$, and so $Z_i\mathbbm{1}_{\sfG} - W_i\mathbbm{1}_{\sfG}$ is bounded above by $C^* (\log n)^3$, and so $Z_i^*\mathbbm{1}_{\sfG} \leq W_i\mathbbm{1}_{\sfG}$ for all $0 \leq i \leq \tau^*$. We also note that $Z_{\tau^*}\mathbbm{1}_{\sfG} \leq C^*(\log n)^3$

Now define $\gt:=\inf\{i\ge 0: Z_i=0\}$, and observe that $K_i \sle Z_i$ for $0\le i \le \tau$.  Furthermore, for $0 \leq i \leq \tau^*$ we can couple $K_i\mathbbm{1}_{\sfG}$ and $W_i\mathbbm{1}_{\sfG}$ using the same coupling as for $W_i\mathbbm{1}_{\sfG}$ and $Z_i\mathbbm{1}_{\sfG}$ with the added stipulation that $1D$ vertices in $K_i\mathbbm{1}_{\sfG}$ behave independently according to their marginals. When $\sfG$ occurs, the number of vertices that are not $1A$ in $K_i$ because they are $1D$ is less than the number of removed vertices in $W_i$ and the number of $1A$ vertices that change to $0A$ by first passing through the $1D$ state in $K_i$ is less the number $1A$ vertices that $W_i$ converts to $0A$ at the start of each one-phase. Thus in this coupling we have
\begin{equation}
Z_i^*\mathbbm{1}_{\sfG} \leq W_i\mathbbm{1}_{\sfG} \leq K_i\mathbbm{1}_{\sfG} \text{ for } 0 \leq i \leq \tau^*
\end{equation}
and so we conclude
\begin{equation}
Z_i^*\mathbbm{1}_{\sfG} \sle K_i\mathbbm{1}_{\sfG}\text{ for }0 \leq i \leq \tau^*,
\end{equation}
and 
\begin{equation}
K_i\mathbbm{1}_{\sfG} \sle Z_i\mathbbm{1}_{\sfG}\text{ for }0 \leq i \leq \tau.
\end{equation}
\end{proof}

Since $Z_i$ and $Z_i^*$ differ by at most $C^*(\log n)^3$, we can now derive upper and lower bounds on $\tau$, which, when combined with this coupling, will yield upper and lower bounds on $\tau_{star}$.
 
We first consider the upper bound on $\tau$. For this we need the following lemma about the transition probabilities of $Z_i$. The intuition is as follows: if a one-phase lasts for a long time, then the properties of $u(t)$ and $v(t)$ in Lemma 5.2 allow us to bound from below the probability that the entire process dies before the next one-phase.
 
\begin{lemma} \label{death-probability-lower-bd}
 For $k, l \geq 1$ if $p(k,l):=\prob{Z_{i+1}=l|Z_i=k}$, then for any $\eta \in (0,1]$ if  $C_1=e\hat\gl(\ga-\gc_1)/(\gc_2-\gc_1)$, then $p(\eta n,0) \ge (1+o(1))(C_1n)^{-1/\gc_1}$.
 \end{lemma}
 
\begin{proof}
 From the definition of the Markov chain $\{Z_i\}_{t\le\gt}$ it is easy to see that
\begin{align*}
p(k,l) & := \E_T \E_{X, Y|T}  \prob{\left.N=X + Y - l\right|T,  X, Y} \\
 & = \E_T \E_{X, Y|T} \hat\lambda(1-\hat\lambda)^{X+Y - l} \mathbf 1_{X+Y \ge l}, \\
p(k,0) & := \sum_{l \le 0} \E_T \E_{X, Y|T}  \prob{\left.N=X + Y - l\right|T,  X, Y} \\
 & = \sum_{l \leq 0} \E_T \E_{X, Y|T} \hat\lambda(1-\hat\lambda)^{X+Y - l}. 
 \end{align*}
So, using the fact that 
 \beq \label{BinLap}
 E\left[s^{\text{Bin}(k,p)}\right] = (1-p(1-s))^k \text{ for $s\in [0, 1]$},\eeq
 and writing $k=\eta n$,  
 \beqax  
 p(\eta n,0) & = & \E_T \left[1-\hat\lambda u(T)\right]^{\gh n}\left[1-\hat\lambda v(T)\right]^{(1-\gh)n} \\
 & = &  \int_0^\infty \left[(1-\hat\lambda u(t))^\eta (1-\hat\lambda v(t))^{1-\eta}\right]^n e^{-t}\; dt.\eeqax
 
To bound the above integral from below, let 
\[t_\eps = \frac {1}{\gamma_1}\log\frac 1\eps \text{ be so that } \exp(-\gamma_1 t_\eps)=\eps.\] 
From property 1.~and 3.~of Lemma \ref{uvprop},  
\[p(\eta n,0) \ge \int_{t_\eps}^\infty (1-\hat\lambda u(t))^ne^{-t} \; dt \ge (1-\hat\lambda u(t_\eps))^n \exp(-t_\eps) 
= (1-c_1\eps-c_2\eps^{\gamma_2/\gamma_1})^n\eps^{1/\gamma_1},\]
where $c_1=\hat\lambda(\gamma_2-1)/(\gamma_2-\gamma_1)$ and $c_2=\hat\lambda(1-\gamma_1)/(\gamma_2-\gamma_1)$. Since $\gamma_2>\gamma_1$, we ignore $\eps^{\gamma_2/\gamma_1}$ term and choose $\eps$ to maximize $(1-c_1\eps)^n\eps^{1/\gamma_1}$. 
In order to do that, we set the derivative of the $\log[(1-c_1\eps)^n\eps^{1/\gamma_1}]$  with respect to $\eps$ to 0 to have
\[ n\frac{c_1}{1-c_1\eps} =\frac{1}{\gamma_1\eps}, \text{ which gives } \eps=(c_1+c_1\gamma_1n)^{-1}.\]
Plugging this value of $\eps$,
\[p(\eta n,0) \ge \left[1-(1+\gamma_1 n)^{-1}-c_2(c_1+c_1\gamma_1 n)^{-\gamma_2/\gamma_1}\right]^n (c_1+c_1\gamma_1 n)^{-1/\gamma_1}
= (c_1e\gamma_1 n)^{-1/\gamma_1}(1+o(1)).\]
\end{proof}

We can now prove the upper bound for $\tau$.

\begin{prop}\label{survival-time-ub}
For the Markov chain $\{Z_i\}$ suppose $\gt=\inf\{t\ge 0: Z_i=0\}$. Fix $\epsilon > 0$. Then there exist constants $N$ and $C$ depending on $\lam$ and $\al$ such that for all $n \geq N$,
\[ \prob{\gt \le C n^{1/\gc_1}} > 1-\epsilon. \]
\end{prop}

\begin{proof}
Fix $\gh_0 \in (0, \hat\gl)$ and let $C_1$ be the constant in Lemma \ref{death-probability-lower-bd}. From part 1 of Lemma \ref{death-probability-lower-bd}, $\gt$ is stochastically dominated by a Geometric random variable with success probability $(1+o(1))(C_1 n)^{-1/\gc_1}$. Hence, for $k\ge1$,
\[ \prob{\gt>k(C_1n)^{1/\gc_1}} \leq \left[1-(1+o(1))(C_1 n)^{-1/\gc_1}\right]^{k(C_1n)^{1/\gc_1}} \leq e^{-(1+o(1))k}.\]
Now choose $N$ so that the $o(1)$ term in the exponent is smaller than $1$ for all $n\ge N$. Choosing $k$ sufficiently large, and setting $C= k C_1^{1/\gamma_1}$ completes the proof.
\end{proof}

Next we consider the lower bound on $\tau$. Lemma~\ref{death-probability-upper-bound} complements Lemma~\ref{death-probability-lower-bd} by providing a matching-order upper bound on the probability of the infection dying during a one-phase.  Lemma~\ref{supermart} will imply that the infection is exponentially unlikely (in the number of infected leaves) to die out in the zero-phase.

\begin{lemma} \label{death-probability-upper-bound}
For $k, l \geq 1$ if $p(k,l):=\prob{Z_{i+1}=l|Z_i=k}$ and $p(k,\leq l):=\sum_{l'\leq l} p(k,l')$, then for any $\gee, \gh_0>0$ satisfying 
 \[\frac{2\gee}{\hat\gl}\log\frac{1}{1-\hat\gl} \le \gh_0 < \hat\gl,\]
   there is a constant $C_2=(1/\hat\gl\gh_0)\log(1/(1-\hat\gl))>0$ such that 
   \[p(\gh n, \le\gee n) \le 3(C_2\gee)^{1/\gc_1} \text{ for any $\gh\ge\gh_0$.}\] 
 \end{lemma}

\begin{proof}
Suppose $s_\eps$ is such that 
\beq \label{sgeedef} 
\eta_0 \exp(-\gc_1s_\gee) = \frac{2\gee}{\hat\gl}\log\frac{1}{1-\hat\gl}.\eeq 
Then $s_\eps \in (0,\infty)$ by our hypothesis about $\eta_0$.

Now, it can be checked that the coefficient of $e^{-\gc_2t}$ in $\gh u(t)+(1-\gh)v(t)$ is negative for $\gh<\gl/(\gl+1-\gc_1)$. Hence,  the coefficient of $e^{-\gc_2t}$ in $\gh_0 u(t)+(1-\gh_0)v(t)$ is negative , as $\gh_0<\gl/(\gl+\ga+1)$.  So using the inequality $\gc_2>\gc_1$, we get $\gh_0 u(t) + (1-\gh_0)v(t) >\gh_0 e^{-\gc_1t}$. Combining this with \eqref{sgeedef} and the fact that $\eta \mapsto \eta u(t)+(1-\eta)v(t)$ is increasing in $\gh$ (by property 1.~of Lemma \ref{uvprop}),
\beq \label{sgeebd}
 \eta u(t)+(1-\eta)v(t) \ge \frac{2\gee}{\hat\gl}\log\frac{1}{1-\hat\gl} \text{ for any } \gh\ge\gh_0 \text{ and } t \le s_\gee.\eeq
 
Now note that
\[p(\gh n,\le \gee n)=\E_T \E_{X,Y|T}(1-\hat\lambda)^{(X+Y-\gee n)^+}
=\int_0^\infty e^{-t}\E_{X,Y|T=t}(1-\hat\lambda)^{(X+Y-\gee n)^+}\, dt.\]
Let $A$ be the event $\{X+Y\ge \gee n\}$. Then the quantity inside the expectation equals $(1-\hat\lambda)^{X+Y-\gee n}+\mathbf 1_{A^c}$.  
Then, splitting the integral in the last display into two parts based on whether $t<s_\gee$ or not and using the fact that the integrand is atmost 1, we get
\[  p(\eta n, \eps n) \le \int_0^{s_\eps} e^{-t} \E_{X,Y|T=t}[(1-\hat\lambda)^{X+Y-\eps n}+\mathbf 1_{A^c}] \; dt+\exp(-s_\eps) .\]
Using Markov inequality
\[ \E_{X,Y|T=t} \mathbf 1_{A^c} \le \E_{X,Y|T=t}(1-\hat\gl)^{X+Y-\gee n}.\]
Also using \eqref{BinLap} and the inequality $1-x \le e^{-x}$,
\[ \E_{X,Y|T=t}(1-\hat\gl)^{X+Y-\gee n}
\le (1-\hat\gl)^{-\gee n}\exp\left[-\hat\lambda n(\eta u(t)+(1-\eta)v(t)\right].\]
Combining the last three displays and using \eqref{sgeebd},
\begin{align*}
p(\gh n, \gee n) & \le  2\int_0^{s_\eps} e^{-t} (1-\hat\gl)^{-\gee n}\exp\left[-\hat\lambda n(\eta u(t)+(1-\eta)v(t)) \right]\; dt + \exp(-s_\gee) \\
&\le 2\int_0^{s_\eps} e^{-t} (1-\hat\gl)^{\gee n} \; dt + \exp(-s_\gee) \\
&\le 2(1-\hat\gl)^{\gee n}+ \exp(-s_\gee).
\end{align*}
From \eqref{sgeedef}, $\exp(-s_\gee) = (c\gee)^{1/\gc_1}$ for an appropriate constant $c$. This proves the assertion.
 \end{proof}

\begin{lemma}\label{supermart}
Suppose $\eta_0 \in (0,\hat\gl)$ and $\tilde\gt:=\inf\{t\ge 0: n^{-1}Z_i \not\in (0,\gh_0)\}$. Then there is a $\gth>0$ such that $U_t:=\exp(-\gth Z_{t\wedge\tilde\gt})$ is a supermartingale.
\end{lemma}

\begin{proof}
Suppose $Z_0=\gh n$ for some $\gh \in (0,\gh_0)$. Define
\[\gf_\gh(\gth) := \left[\E\left(\left.\exp(-\gth Z_1)\right| Z_0=\gh n\right)\right]^{1/n} - e^{-\gth\gh}.\]
Clearly $\gf_\gh(0)=0$ and $\gf_\gh \in C^1[0, \infty)$  for any $\gh>0$. 
We will show 
\beq\label{derivativebd}
 (a)\;\; \gf'_{\gh_0}(0) <0 \quad \text{ and } \quad (b)\;\; \gf'_{\gh}(0) \text{ is an increasing function of $\gh$.}\eeq
  Using continuity of $\gf'_{\gh_0}$  (a) will imply  that there exists $\gth >0$ such that $\gf'_{\gh_0}(\gb)<0$ for all $\gb \in [0,\gth]$. Also using the mean value theorem, $\gf_\gh(\gth)=\gf'_\gh(\gb_0) \gth$ for some $\gb_0 \in [0,\gth]$. Then (b) will imply  $\gf'_\gh(\gb_0) \le \gf'_{\gh_0}(\gb_0) <0$ for $\gh\le\gh_0$, which in turn implies $\gf_\gh(\gth)<0$ for $\gh\le\gh_0$.
  
  In order to show \eqref{derivativebd} we will find an expression for $\gf_\gh(\gth)$. Clearly,
  \[ \gf_\gh(\gth) =\left[\E_T\E_{X,Y|T}\E_{N|X,Y,T} \exp(-\gth(X+Y-N)^+)\right]^{1/n}-e^{-\gth\gh},\]
  where $T\sim Exponential(1)$, given $T=t$ $X\sim\text{Bin}(\gh n,u(t)), Y\sim\text{Bin}((1-\gh) n,v(t))$ and $N$ is as in \eqref{Ndef}. Now
  \beqax
  \E_{N|X,Y,T} \exp(-\gth(X+Y-N)^+) & = & \E_{N|X,Y,T} [\exp(-\gth(X+Y-N)) \mathbf 1_{\{N<X+Y\}}] + \E_{N|X,Y,T} \mathbf 1_{\{N\ge X+Y\}} \\
  & = & e^{-\gth(X+Y)} \sum_{j=0}^{X+Y-1}\hat\gl[e^\gth(1-\hat\gl)]^j +(1-\hat\gl)^{X+Y} \\
  & = & \frac{\hat\gl}{1-e^\gth(1-\hat\gl)}\left[e^{-\gth(X+Y)} - (1-\hat\gl)^{X+Y}\right] + (1-\hat\gl)^{X+Y}.
  \eeqax
  Therefore, using \eqref{BinLap}
  \[ \gf_\gh(\gth) = \left(\int_0^\infty e^{-t}  \left[\frac{\hat\gl}{1-e^\gth(1-\hat\gl)}(\gx(\gh, e^{-\gth}, t) - \gx(\gh, 1-\hat\gl, t)) + \gx(\gh, 1-\hat\gl, t) \right]\; dt\right)^{1/n} - e^{-\gth\gh},\]
  where
  \[ \gx(\gh,s,t) := \left[(1-(1-s)u(t))^\gh (1-(1-s)v(t))^{1-\gh}\right]^n.\]
Since $\gx(\gh,1,t)=1$,
\[ \gf'_\gh(0)   =  \frac 1n \int_0^\infty e^{-t} \left(\frac{1-\hat\gl}{\hat\gl}[1-\gx(\gh,1-\hat\gl,t)] +\left. \frac{d}{d\gth}\gx(\gh,e^{-\gth}, t)\right|_{\gth=0}\right)\; dt +\eta.\]
The first integrand is an increasing function of $\eta$, as $u(t)>v(t)$ by property 3.~of Lemma \ref{uvprop}. On the other hand, the second integrand is $n(-\gh u(t)-(1-\gh)v(t))$, and hence (b) of \eqref{derivativebd} holds by property 4.~of Lemma \ref{uvprop}. Also the first integrand is  at most 1, so
\[ \gf'_{\gh_0}(0) \le  \frac 1n - \int_0^\infty e^{-t}(\gh_0 u(t) + (1-\gh_0) v(t)) \; dt +\eta < 0\]
using property 4.~of Lemma \ref{uvprop} and the fact that $\gh_0<\hat\gl$. This proves (a) of \eqref{derivativebd} and proof of the lemma is complete.
\end{proof}

We can now prove the lower bound on $\tau$.

\begin{prop}\label{survival-time-lb}
For the Markov chain $\{Z_i\}$ suppose $\gt=\inf\{t\ge 0: Z_i=0\}$. Fix $\epsilon > 0$. Then there exist $N, K_o$ such that for all $n \geq n$,
\[ \prob{\frac{1}{K_0} \left(\frac{n}{\log (n)^4}\right)^{1/\gc_1} \le \gt} > 1 - \epsilon \] 
\end{prop}
 
\begin{proof}
Let $C_2$ and $\gth$ be the constants in Lemmas \ref{death-probability-upper-bound} and  \ref{supermart}. Divide the interval $[0, n]$  into three parts
\[ I_1 :=[0, (\gc_1\gth)^{-1}C(\log n)^4), \quad I_2 := [(\gc_1\gth)^{-1}C(\log n)^4, \gh_0n], \quad I_3 := (\gh_0n, n]\]
and note that so long as $Z_i$ is in $I_2$ or $I_3$ then the process $Z_i^*$ defined in lemma~\ref{couple} is greater than $0$.
Using $\gee=C\log(n)^4/(\gc_1\gth n)$ in Lemma \ref{death-probability-upper-bound}, it is easy to see that the number of times $Z_i$ avoids jumping from $I_3$ to $I_1$ stochastically dominates a Geometric random variable with success probability $C(\log(n)^4/n)^{1/\gc_1}$ for some constant $C>0$.

Also, if $Z_0\in I_2$, then applying the optional stopping theorem for the 
stopping time \[\tilde\gt := \inf\{t\ge 0: Z_i \not \in (0,\gh_0n)\},  \text{ and supermartingale  } U_t := \exp(-\gth Z_{t}),  0 \le t\le \tilde\gt,\]
we see that if $q := \prob{Z_{\tilde\gt}=0}$, then
\[ q \le \E U_{\tilde\gt} \le U_0 \le n^{-1/\gc_1}.\]  
So, the number of times $Z_i$ jumps from $I_2$ to $I_3$ stochastically dominates a Geometric random variable with success probability $n^{-1/\gc_1}$ . Combining these two observations, $\gt$ stochastically dominates sum of two  Geometric random variables with success probability $C(\log(n)^4/n)^{1/\gc_1}$.
Hence 
\[\prob{\gt<K^{-1} (n/\log (n)^4)^{1/\gc_1} } \le  2\Big(1-\Big[1-C\big(\log(n)^4/n\big)^{1/\gc_1}\Big]\Big)^{(n/\log(n)^4)^{1/\gc_1}/(2K)} \le C/K \to 0
\]
as $K\to\infty$.
\end{proof}

We now are ready to finish the proof of Theorem~\ref{star-thm}.
\begin{proof}[Proof of Theorem~\ref{star-thm}]
Propositions~\ref{survival-time-ub} and~\ref{survival-time-lb} give bounds on $\tau$ and so it remains to compare $\tau$ and $\tau_{star}$.

Let $\tau_K = \inf\{t \geq 0: K_t = 0\}$ for the Markov chain $(K_i,L_i)$ (without assuming $L_i = 0$ for all $i$). From Lemmas~\ref{tight} and \ref{couple}, the good event $\sfG$ has probability at least $1 - \epsilon$, $Z_i^*\mathbbm{1}_{\sfG} \sle K_i\mathbbm{1}_{\sfG} \sle Z_i\mathbbm{1}_{\sfG}$ for $0 \leq i \leq \tau^*$, and $K_i \sle Z_i$ for $0 
\leq i \leq \tau$. Furthermore, from Proposition \eqref{survival-time-lb}, we see that $Z_i^* > 0$ so long as $Z_i$ is not in the interval $I_1$. From this, we observe that the same holds for $K_i$ and that we do not reach $\tau^*$ until $Z_i$ jumps to the interval $I_1$, and so it is sufficient for the coupling to hold until time $\tau^*$ for the lower bound. Thus we can conclude  
\begin{equation}\label{tauk}
\begin{aligned}
&\prob{\frac{1}{K_0} \left(\frac{n}{\log (n)^4}\right)^{1/\gc_1} \leq \tau_K \leq K_0(C_1n)^{1/\gc_1}} \\
&\geq \prob{\left\{\frac{1}{K_0} \left(\frac{n}{\log (n)^4}\right)^{1/\gc_1} \leq \tau \leq K_0(C_1n)^{1/\gc_1}\right\} \bigcap \left\{\sfG^c\right\}}\\
&\geq 1 - \prob{\left\{\frac{1}{K_0} \left(\frac{n}{\log (n)^4}\right)^{1/\gc_1} \leq \tau \leq K_0(C_1n)^{1/\gc_1}\right\}^c} - \prob{\sfG^c} \\
&> 1 - 2\epsilon.
\end{aligned}
\end{equation}

However, $\tau_K$ counts the number of one-phases until the infection dies, and so we must bound the total amount of time this takes. The length of a one-phase is an Exp$(1)$ random variable and does not depend on the states of the vertices other than the center, so clearly $\sum_{i=1}^{\tau_K}X_i \sle \tau_{star}$ where the $X_i$ are iid Exp$(1)$ random variables gives a lower bound on the time to extinction. The distribution of the length of a zero-phase depends on the number of 1A nodes present at the start of the zero-phase. However, the length of a zero-phase is dominated by an Exp$(\min(\lambda,1))$ random variable for any configuration. $\tau_K$ gives the number of zero-phases before the process dies, and let $Z$ be the length of the last zero-phase. Then $\tau_{star} \sle \sum_{i=1}^{\tau_K}(X_i + Y_i)$ where the $X_i$ are as before and the $Y_i$ are iid Exp$(\min(\lambda,1))$ random variables gives an upper bound.

Using large deviation bounds for all $m$ we have 
$$\prob{\sum_{i = 1}^{m}X_i \leq m/2} \leq e^{-m/6},$$
$$\prob{\sum_{i = 1}^{m}X_i + Y_i \geq 2m(1 + \frac{1}{\min(\lam,1)})} \leq 2e^{-m/\frac{6}{\min(\lam,1)}}.$$

So then to get a lower bound we observe
\begin{equation}
\begin{aligned}
& \prob{\tau_{star} \geq \tau_K/2} \geq \prob{\sum_{i=1}^{\tau_K}X_i \geq \tau_K/2}\\
& \geq \prob{\{ \tau_K \geq M \} \bigcap \{ \sum_{i=1}^{\tau_K}X_i \geq \tau_K/2\}}\\
& \geq \prob{\{ \tau_K \geq M \} \bigcap \{ \sum_{i=1}^{m}X_i \geq m/2 \hspace{2mm} \forall m \geq M\}}\\
& \geq 1 - \prob{\tau_K < M} - \sum_{m=M}^{\infty}\prob{\sum_{i=1}^m X_i \geq m/2}\\
& = 1 - \prob{\tau_K < M} - \frac{e^{(1-M)/6}}{e^{1/6}-1}
\end{aligned} 
\end{equation}
Choosing $M = \frac{1}{K_0}(\frac{n}{\log (n)^4})^{1/\gamma_1}$ and using \eqref{tauk}, we conclude that there exists $N$ such that for all $n \geq N$
\begin{equation}
\prob{\tau_{star} \geq \tau_K/2} \geq 1 - 2\epsilon,
\end{equation}
giving a lower bound on the survival time in terms of the number of cycles

To get an upper bound we observe
\begin{equation}
\begin{aligned}
& \prob{\tau_{star} \leq 2(1 + \frac{1}{\min(\lam,1)})\tau_K} \geq \prob{\sum_{i = 1}^{\tau_K}X_i + Y_i \leq 2(1 + \frac{1}{\min(\lam,1)})\tau_K}\\
& \geq \prob{\{\tau_K \geq M\} \bigcap \{\sum_{i = 1}^{\tau_K}X_i + Y_i \leq 2(1 + \frac{1}{\min(\lam,1)})\tau_K \}}\\
& \geq \prob{\{\tau_K \geq M\} \bigcap \{\sum_{i = 1}^{m}X_i + Y_i \leq 2(1 + \frac{1}{\min(\lam,1)})m \hspace{2mm} \forall m \geq M\}}\\
& \geq 1 - \prob{\tau_K < M} - \sum_{j=M}^{\infty}\prob{\sum_{i = 1}^{m}X_i + Y_i \leq 2(1 + \frac{1}{\min(\lam,1)})m}\\
& = 1 - \prob{\tau_K < M} + \frac{e^{(1-M)/\frac{6}{\min(\lambda,1)}}}{e^{1/(\min(\lambda,1))}-1}
\end{aligned}
\end{equation}

Again choosing $M = \frac{1}{K_0}(\frac{n}{\log (n)^4})^{1/\gamma_1}$ and using \eqref{tauk}, we conclude that there exists $N$ such that for all $n \geq N$
\begin{equation}
\prob{\tau_{star} \leq 2(1+\frac{1}{\min(\lam,1)})\tau_K} \geq 1 - 2\epsilon,
\end{equation}
giving an upper bound on the survival time in terms of the number of cycles.

Combining our comparison of $\tau$ and $\tau_K$ with our comparison of $\tau_K$ and $\tau_{star}$, we conclude for any $\lambda, \alpha >0$ and any $\epsilon > 0$, there exist $C, K$, and $N$ depending on $\lambda$ and $\alpha$ such that for all $n \geq N$,
$$\prob{\frac{1}{K}(\frac{n}{\log (n)^4})^{1/\gamma} \leq \tau_{star} \leq K(Cn)^{1/\gamma}} > 1 - \eps.$$
\end{proof}

\section*{Acknowledgements}
SC and DS are grateful to the Mathematical Biosciences Institute, where some of this work was conducted. DS and MW were partially supported by the NSF grant CCF--1740761.

\bibliographystyle{alpha}
\bibliography{citations}

\end{document}